\documentclass[12pt,reqno]{amsart}
\usepackage{amsmath,amsthm,amsfonts,amssymb,mathrsfs}
\input xymatrix
\xyoption{all}
\date{\today}

\usepackage{color}

\usepackage{hyperref}

\newtheorem{theorem}{Theorem}[section]
\newtheorem{proposition}[theorem]{Proposition}
\newtheorem{corollary}[theorem]{Corollary}
\newtheorem{lemma}[theorem]{Lemma}
\theoremstyle{definition}

\newtheorem{example}[theorem]{Example}%[section]
\newtheorem{remark}[theorem]{Remark}%[section]

\newtheorem{definition}[theorem]{Definition}%[section]

\begin{document}

\title[On pseudocompact topological Brandt $\lambda^0$-extensions of semitopological monoids]{On pseudocompact topological Brandt $\lambda^0$-extensions of semitopological monoids}

\author[O. Gutik]{Oleg~Gutik}
\address{Department of Mechanics and Mathematics, National University of Lviv, Universytetska 1, Lviv, 79000, Ukraine}
\email{o\_\,gutik@franko.lviv.ua, ovgutik@yahoo.com}

\author[K. Pavlyk]{Kateryna~Pavlyk}
\address{Institute of Mathematics, University of Tartu, J. Liivi 2, 50409, Tartu, Estonia}
\email{kateryna.pavlyk@ut.ee}



\begin{abstract}
In the paper we investigate topological properties of a topological Brandt
$\lambda^0$-extension $B^0_{\lambda}(S)$ of a semitopological monoid $S$ with zero. In particular we prove that for every Tychonoff pseudocompact (resp., Hausdorff countably compact, Hausdorff compact) semitopological monoid $S$ with zero there exists a unique semiregular pseudocompact (resp., Hausdorff countably compact, Hausdorff compact) extension $B^0_{\lambda}(S)$ of $S$ and establish theirs Stone-\v{C}ech and Bohr compactifications. We also describe a category whose objects are ingredients in the constructions of pseudocompact (resp., countably compact, sequentially compact, compact) topological Brandt $\lambda^0$-extensions of pseudocompact (resp., countably compact, sequentially compact, compact) semitopological monoids with zeros.

\end{abstract}

\maketitle

%\tableofcontents

\section{Introduction, preliminaries and definitions}

In this article we shall follow the terminology of \cite{BucurDeleanu1968, CliffordPreston1961-1967, Engelking1989, Howie1995, Petrich1984, Ruppert1984}. All topological spaces are assumed to be Hausdorff. By $\omega$ we shall denote the first infinite ordinal and by $|A|$ the cardinality of the set $A$. All cardinals we shall identify with their corresponding initial ordinals. If $Y$ is a subspace of a topological space $X$ and $A\subseteq Y$, then by $\operatorname{cl}_Y(A)$ and $\operatorname{int}_Y(A)$ we shall denote the topological closure and the interior of the subset $A$ in $Y$, respectively.

A \emph{semigroup} is a non-empty set with a binary associative
operation. A semigroup $S$ is called \emph{inverse} if for any $x\in
S$ there exists a unique $y\in S$ such that $x\cdot y\cdot x=x$ and
$y\cdot x\cdot y=y$. Such an element $y$ in $S$ is called the
\emph{inverse} of $x$ and denoted by $x^{-1}$. The map defined on an
inverse semigroup $S$ which maps every element $x$ of $S$ to its
inverse $x^{-1}$ is called the \emph{inversion}.

For a semigroup $S$ by $E(S)$ we denote the subset of all idempotents of $S$.

Let $S$ be a semigroup with zero $0_S$ and $\lambda$ be any cardinal $\geqslant 1$. We define the semigroup operation on the set $B_{\lambda}(S)=(\lambda\times S\times
\lambda)\cup\{ 0\}$ as follows:
\begin{equation*}
 (\alpha,a,\beta)\cdot(\gamma, b, \delta)=
  \begin{cases}
    (\alpha, ab, \delta), & \text{ if~} \beta=\gamma; \\
    0,                    & \text{ if~} \beta\ne \gamma,
  \end{cases}
\end{equation*}
and $(\alpha, a, \beta)\cdot 0=0\cdot(\alpha, a, \beta)=0\cdot
0=0,$ for all $\alpha, \beta, \gamma, \delta\in\lambda$ and
$a,b\in S$. If $S=S^1$ then the semigroup $B_\lambda(S)$ is called
the {\it Brandt $\lambda$-extension of the semigroup}
$S$~\cite{Gutik1999, GutikPavlyk2001}. It is clear that ${\mathcal
J}=\{ 0\}\cup\{(\alpha, 0_S, \beta)\mid 0_S \hbox{~is the zero
of~}S\}$ is an ideal of $B_\lambda(S)$. We put
$B^0_\lambda(S)=B_\lambda(S)/{\mathcal J}$ and the semigroup
$B^0_\lambda(S)$ is called the {\it Brandt $\lambda^0$-extension
of the semigroup $S$ with zero}~\cite{GutikPavlyk2006}.

Next, if $A\subseteq S$ then we shall denote
$A_{\alpha\beta}=\{(\alpha, s,\beta)\mid s\in A \}$ if $A$ does
not contain zero, $A_{\alpha,\beta}=\{(\alpha, s, \beta)\mid s\in
A\setminus\{ 0_S\} \}\cup \{ 0\}$ and
$A_{\alpha,\beta}^*=A_{\alpha,\beta}\setminus\{0\}$ if $0_S\in A$,
for $\alpha, \beta\in{\lambda}$. If $\mathcal{I}$ is a trivial
semigroup (i.e., $\mathcal{I}$ contains only one element), then we
denote the semigroup $\mathcal{I}$ with the adjoined zero by
${\mathcal{I}}^0$. Obviously, for any $\lambda\geqslant 2$, the
Brandt $\lambda^0$-extension of the semigroup ${\mathcal{I}}^0$ is
isomorphic to the semigroup of $\lambda\times\lambda$-matrix units
and any Brandt $\lambda^0$-extension of a semigroup with zero
which also has a non-zero idempotent contains the semigroup of
$\lambda\times\lambda$-matrix units.

We shall denote the semigroup of $\lambda\times\lambda$-matrix
units by $B_\lambda$ and the subsemigroup of
$\lambda\times\lambda$-matrix units of the Brandt
$\lambda^0$-extension of a monoid $S$ with zero by
$B^0_\lambda(1)$. We always consider the Brandt
$\lambda^0$-extension only of a monoid with zero. Obviously, for
any monoid $S$ with zero we have $B^0_1(S)=S$. Note that every
Brandt $\lambda$-extension of a group $G$ is isomorphic to the
Brandt $\lambda^0$-extension of the group $G^0$ with adjoined
zero. The Brandt $\lambda^0$-extension of the group with adjoined
zero is called a \emph{Brandt
semigroup}~\cite{CliffordPreston1961-1967, Petrich1984}. A
semigroup $S$ is a Brandt semigroup if and only if $S$ is a
completely $0$-simple inverse semigroup~\cite{Clifford1942,
Munn1957} (cf.  also \cite[Theorem~II.3.5]{Petrich1984}). We also
observe that the trivial semigroup $\mathcal{I}$ is isomorphic to the Brandt $\lambda^0$-extension of $\mathcal{I}$ for
every cardinal $\lambda\geqslant 1$. We shall say that the Brandt
$\lambda^0$-extension $B_\lambda^0(S)$ of a semigroup $S$ is
\emph{finite} if the cardinal $\lambda$ is finite.

For a topological space $X$, a family $\{A_s\mid
s\in\mathscr{S}\}$ of subsets of $X$ is called \emph{locally
finite} if for every point $x\in X$ there exists an open
neighbourhood $U$ of $x$ in $X$ such that the set
$\{s\in\mathscr{S}\mid U\cap A_s\neq\varnothing\}$ is finite. A
set $A$ of a topological space $X$ is called \emph{regular open}
if $A=\operatorname{int}_X(\operatorname{cl}_X(A))$.

We recall that a topological space $X$ is said to be
\begin{itemize}
  \item \emph{semiregular} if $X$ has a base with regular open subsets;
  \item \emph{compact} if each open cover of $X$ has a finite subcover;
  \item \emph{sequentially compact} if each sequence $\{x_i\}_{i\in\omega}$ of $X$ has a convergent subsequence in $X$;
  \item \emph{countably compact} if each open countable cover of $X$ has a finite subcover;
  \item \emph{countably compact at a subset} $A\subseteq X$ if every infinite subset $B\subseteq A$  has  an  accumulation  point $x$ in $X$;
  \item \emph{countably pracompact} if there exists a dense subset $A$ in $X$  such that $X$ is countably compact at $A$;
  \item \emph{pseudocompact} if each locally finite open cover of $X$ is finite.
\end{itemize}
According to Theorem~3.10.22 of \cite{Engelking1989}, a Tychonoff topological space $X$ is pseudocompact if and only if each continuous real-valued function on $X$ is bounded. Also, a Hausdorff topological space $X$ is pseudocompact if and only if every locally finite family of non-empty open subsets of $X$ is finite. Every compact space and every sequentially compact space are countably compact, every countably compact space is countably pracompact, and every countably pracompact space is pseudocompact (see \cite{Arkhangelskii1992} and \cite{Engelking1989}).

We recall that the Stone-\v{C}ech compactification of a Tychonoff space $X$ is a
compact Hausdorff space $\beta X$ containing $X$ as a dense subspace so that each continuous map $f\colon X\rightarrow Y$ to a compact Hausdorff space $Y$ extends to a continuous
map $\overline{f}\colon \beta X\rightarrow Y$.

A {\it semitopological} (resp., \emph{topological}) {\it semigroup} is a topological space together with a separately (resp., jointly) continuous semigroup operation.

\begin{definition}[\cite{GutikPavlyk2006}]\label{definition-1.1}
Let $\mathscr{S}$ be some class of semitopological semigroups with zero. Let $\lambda$ be any cardinal $\geqslant 1$, and $(S,\tau)\in\mathscr{S}$. Let $\tau_{B}$ be a topology on
$B^0_{\lambda}(S)$ such that:
\begin{itemize}
  \item[a)] $\left(B^0_{\lambda}(S), \tau_{B}\right)\in\mathscr{S}$; \;   and
  \item[b)] $\tau_{B}|_{S_{\alpha,\alpha}}=\tau$ for some $\alpha\in\lambda$.
\end{itemize}
Then $\left(B^0_{\lambda}(S), \tau_{B}\right)$ is called the \emph{topological Brandt $\lambda^0$-extension of $(S, \tau)$ in} $\mathscr{S}$.
\end{definition}

The notion of the topological Brandt $\lambda^0$-extension of topological semigroup was introduced in the paper \cite{GutikPavlyk2006} on purpose of constructing an example of an absolutely $H$-closed semigroup $S$ with an absolutely $H$-closed ideal $I$ such that $S/I$ is not a topological semigroup. In \cite{GutikRepovs2010} Gutik and Repov\v{s} described compact topological Brandt $\lambda^0$-extensions in the class of topological semigroups and countably compact topological Brandt $\lambda^0$-extensions in the class of topological inverse semigroups, and correspondent categories. Also, in \cite{GutikPavlykReiter2009} countably compact topological Brandt $\lambda^0$-extensions in the class of topological semigroups were described. This paper is motivated by the result obtained in \cite{GutikPavlyk2005} which states that every Hausdorff pseudocompact topology $\tau$ on the infinite semigroup of matrix units $B_\lambda$ such that $(B_\lambda,\tau)$ is a semitopological semigroup, is compact.

In this paper we investigate topological properties of a topological Brandt
$\lambda^0$-extension $B^0_{\lambda}(S)$ of a semitopological monoid $S$ with zero. In particular we prove that for every Tychonoff pseudocompact (resp., Hausdorff countably compact, Hausdorff compact) semitopological monoid $S$ with zero there exists a unique semiregular pseudocompact (resp., Hausdorff countably compact, Hausdorff compact) extension $B^0_{\lambda}(S)$ of $S$ and establish their Stone-\v{C}ech and Bohr compactifications. We also describe a category whose objects are ingredients in the constructions of pseudocompact (resp., countably compact, sequentially compact, compact) topological Brandt $\lambda^0$-extensions of pseudocompact (resp., countably compact, sequentially compact, compact) semitopological monoids with zeros.
%%%%%%%%%%%%%%%%%%%%%%%%%%%%%%%%%%%%%%%%%%%%%

\section{Pseudocompact Brandt $\lambda^0$-extensions of semitopological monoids}

The following three lemmas describe the general structure of topological Brandt $\lambda^0$-extensions of semitopological semigroups in the class of semitopological semigroups.

\begin{lemma}\label{lemma-2.1}
Let $\lambda$ be any cardinal $\geqslant 1$ and $\tau^*_{B}$ be a topology on the Brandt $\lambda^0$-extension $B^0_{\lambda}(S)$ of monoid $S$ with zero such that $\left(B^0_{\lambda}(S),\tau^*_{B}\right)$ is a semitopological semigroup. Then $S_{\alpha,\beta}$ and $S_{\gamma,\delta}$ with the induced topology from $\left(B^0_{\lambda}(S),\tau^*_{B}\right)$ are homeomorphic subspaces, and, moreover, $S_{\alpha,\alpha}$ and $S_{\beta,\beta}$ are topologically isomorphic semitopological monoids for all $\alpha,\beta,\gamma,\delta\in\lambda$.
\end{lemma}

\begin{proof}
Since $\left(B^0_{\lambda}(S), \tau^*_{B}\right)$ is a semitopological semigroup we conclude  that the maps $\varphi^{(\gamma,\alpha)}_{(\beta,\delta)}\colon B^0_{\lambda}(S)\rightarrow B^0_{\lambda}(S)$ and $\varphi^{(\alpha,\gamma)}_{(\delta,\beta)}\colon B^0_{\lambda}(S)\rightarrow B^0_{\lambda}(S)$ defined by the formulae $\varphi^{(\gamma,\alpha)}_{(\beta,\delta)}(x)=(\gamma,1_S,\alpha)\cdot x\cdot(\beta,1_S,\delta)$ and $\varphi^{(\alpha,\gamma)}_{(\delta,\beta)}(x)=(\alpha,1_S,\gamma)\cdot x\cdot(\delta,1_S,\beta)$, respectively, are continuous maps and, moreover, the restrictions $\left(\varphi^{(\alpha,\gamma)}_{(\delta,\beta)}\circ \varphi^{(\gamma,\alpha)}_{(\beta,\delta)}\right)\Big|_{S_{\alpha,\beta}}\colon S_{\alpha,\beta} \rightarrow S_{\alpha,\beta}$ and $\left(\varphi^{(\gamma,\alpha)}_{(\beta,\delta)}\circ \varphi^{(\alpha,\gamma)}_{(\delta,\beta)} \right)\Big|_{S_{\gamma,\delta}}\colon S_{\gamma,\delta} \rightarrow S_{\gamma,\delta}$ are identity maps for all $\alpha,\beta,\gamma,\delta\in\lambda$. Hence we have that the restrictions $\varphi^{(\gamma,\alpha)}_{(\beta,\delta)}\Big|_{S_{\alpha,\beta}}\colon S_{\alpha,\beta}\rightarrow S_{\gamma,\delta}$ and $\varphi^{(\alpha,\gamma)}_{(\delta,\beta)}\Big|_{S_{\gamma,\delta}}\colon S_{\gamma,\delta}\rightarrow S_{\alpha,\beta}$ are homeomorphisms. Now, simple verifications imply that the map $\varphi^{(\beta,\alpha)}_{(\alpha,\beta)}\Big|_{S_{\alpha,\alpha}}\colon S_{\alpha,\alpha}\rightarrow S_{\beta,\beta}$ is an isomorphism of subsemigroups $S_{\alpha,\alpha}$ and $S_{\beta,\beta}$ of $B^0_{\lambda}(S)$. This completes the proof of the lemma.
\end{proof}

\begin{lemma}\label{lemma-2.2}
Let $\lambda$ be any cardinal $\geqslant 1$ and $\tau^*_{B}$ be a topology on the Brandt $\lambda^0$-extension $B^0_{\lambda}(S)$ of monoid $S$ with zero such that $\left(B^0_{\lambda}(S),\tau^*_{B}\right)$ is a Hausdorff semitopological semigroup. Then for arbitrary $(\alpha,s,\beta)\in B^0_{\lambda}(S)\setminus\{0\}$ there exists an open neighbourhood $U_{(\alpha,s,\beta)}$ of the point $(\alpha,s,\beta)$ in $\left(B^0_{\lambda}(S),\tau^*_{B}\right)$ such that $U_{(\alpha,s,\beta)}\subseteq S_{\alpha,\beta}^*$, and hence $S_{\alpha,\beta}^*$ is an open subset and $S_{\alpha,\beta}$ is a closed subset of $\left(B^0_{\lambda}(S),\tau^*_{B}\right)$, for all $\alpha,\beta\in\lambda$.
\end{lemma}

\begin{proof}
Let $V_{(\alpha,s,\beta)}$ be an open neighbourhood of the point $(\alpha,s,\beta)$ in $\left(B^0_{\lambda}(S),\tau^*_{B}\right)$ such that $0\notin V_{(\alpha,s,\beta)}$. Then the continuity of the map $\varphi^{(\alpha,\alpha)}_{(\beta,\beta)}\colon B^0_{\lambda}(S)\rightarrow B^0_{\lambda}(S)$ which is defined by the formula $\varphi^{(\alpha,\alpha)}_{(\beta,\beta)}(x)=(\alpha,1_S,\alpha)\cdot x\cdot(\beta,1_S,\beta)$ implies that there exists an open neighbourhood $U_{(\alpha,s,\beta)}$ of the point $(\alpha,s,\beta)$ in $\left(B^0_{\lambda}(S),\tau^*_{B}\right)$ such that $U_{(\alpha,s,\beta)}\subseteq S_{\alpha,\beta}^*$. Hence $S_{\alpha,\beta}^*$ is open in $\left(B^0_{\lambda}(S),\tau^*_{B}\right)$ and since $S_{\alpha,\beta}=B^0_{\lambda}(S)\setminus \bigcup_{(\gamma,\delta)\in(\lambda\times\lambda)\setminus(\alpha,\beta)} S_{\gamma,\delta}^*$ we have that $S_{\alpha,\beta}$ is a closed subset of the topological space $\left(B^0_{\lambda}(S),\tau^*_{B}\right)$.
\end{proof}

\begin{proposition}\label{proposition-2.3}
Let $S$ be a semitopological monoid with zero, $\lambda\geqslant
1$ be any cardinal, and $B_\lambda^0(S)$ be a topological Brandt
$\lambda^0$-extension of $S$ in the class of semitopological
semigroup. Then every non-trivial continuous homomorphic image of
$B_\lambda^0(S)$ in a semitopological semigroup $W$ is a
topological Brandt $\lambda^0$-extension of some semitopological
monoid with zero. Moreover, if $T$ is the continuous image of
$B_\lambda^0(S)$ under a homomorphism $h$, then $T$ is
topologically isomorphic to a topological Brandt
$\lambda^0$-extension of the continuous homomorphic image of the
semitopological monoid $S_{\alpha,\alpha}$ under the homomorphism
$h$ for any $\alpha\in \lambda$. Also, the subspaces
$h(A_{\alpha\beta})$ and $h(A_{\gamma\delta})$ of $W$ are
homeomorphic for all $\alpha,\beta,\gamma,\delta\in\lambda$, and
every nonempty $A\subseteq S$.

\end{proposition}

\begin{proof}
Let $h\colon B_\lambda^0(S)\rightarrow W$ be a continuous homomorphism of a topological Brandt
$\lambda^0$-extension of a semitopological monoid $S$  into a semitopological semigroup $W$. The algebraic part of the proof follows from Proposition~3.2 of \cite{GutikRepovs2010}.
If $h$ is an annihilating homomorphism, then the statement of the
lemma is trivial.

We fix arbitrary $\alpha, \beta, \gamma, \delta\in{\lambda}$ and define the maps
$\varphi^{(\gamma,\alpha)}_{(\beta,\delta)}\colon W\rightarrow W$ and $\varphi^{(\alpha,\gamma)}_{(\delta,\beta)}\colon W\rightarrow W$ by the formulae
\begin{equation*}
\varphi^{(\gamma,\alpha)}_{(\beta,\delta)}(s)=h((\gamma, 1,
\alpha))\cdot s\cdot h((\beta, 1, \delta))
 \quad \text{~and~} \quad
\varphi^{(\alpha,\gamma)}_{(\delta,\beta)}(s)=h((\alpha, 1,
\gamma))\cdot s\cdot h((\delta, 1, \beta)),
\end{equation*}
for $s\in W$. Obviously we have that
\begin{equation*}
\varphi^{(\alpha,\gamma)}_{(\delta,\beta)}\left(
\varphi^{(\gamma,\alpha)}_{(\beta,\delta)}\left(h\left((\alpha, x,
\beta)\right) \right) \right)=h\left((\alpha, x, \beta)\right)
 \quad \text{~and~} \quad
\varphi^{(\gamma,\alpha)}_{(\beta,\delta)}\left(
\varphi^{(\alpha,\gamma)}_{(\delta,\beta)}\left(h\left((\gamma, x,
\delta)\right) \right) \right)=h\left((\gamma, x, \delta)\right),
\end{equation*}
for all $\alpha,\beta,\gamma,\delta\in\lambda$, $x\in S$, and hence the map $\varphi^{(\gamma,\alpha)}_{(\beta,\delta)}\Big|_{h(S_{\alpha,\beta})}\colon h(S_{\alpha,\beta})\rightarrow h(S_{\gamma,\delta})$ is invertible to the map
$\varphi^{(\alpha,\gamma)}_{(\delta,\beta)}\Big|_{h(S_{\gamma,\delta})}\colon h(S_{\gamma,\delta})\rightarrow h(S_{\alpha,\beta})$. This implies that the map $\varphi^{(\gamma,\alpha)}_{(\beta,\delta)}\Big|_{h(S_{\alpha,\beta})}\colon h(S_{\alpha,\beta})\rightarrow h(S_{\gamma,\delta})$ is a homeomorphism and hence the last statement of the proposition holds.
\end{proof}

Results of Section~2 of \cite{GutikPavlyk2006} imply that for any infinite cardinal $\lambda$ and every non-trivial topological semigroup (and hence for a semitopological semigroup) $S$, there are many topological Brandt $\lambda^0$-extensions of $S$ in the class of semitopological semigroups. Moreover, for any infinite cardinal $\lambda$ on the Brandt $\lambda^0$-extension of two-element monoid with zero (i.e., on the infinite semigroup of $\lambda\times\lambda$-units) there exist many topologies which turns $B_\lambda$ into a topological (and hence a semitopological) semigroup (cf.  \cite{GutikPavlyk2005}).

The following proposition describes the properties of finite topological Brandt $\lambda^0$-extensions of semitopological semigroups in the class of semitopological semigroups.

\begin{proposition}\label{proposition-2.4}
Let $\lambda$ be any finite non-zero cardinal. Let $(S,\tau)$ be a
semitopological semigroup and $\tau_{B}$ a topology on
$B^0_{\lambda}(S)$ such that $\left(B^0_{\lambda}(S),
\tau_{B}\right)$ is a semitopological semigroup and
$\tau_{B}|_{S_{\alpha,\alpha}}=\tau$ for some $\alpha\in\lambda$.
Then the following assertions hold:
\begin{itemize}
    \item[$(i)$] If a non-empty subset $A\not\ni 0_S$ of $S$ is     open in $S$, then so is $A_{\beta,\gamma}$ in $\left(B^0_{\lambda}(S),\tau_{B}\right)$ for any
        $\beta,\gamma\in\lambda$;

    \item[$(ii)$] If a non-empty subset $A\ni 0_S$ of $S$ is open in $S$, then so is $\bigcup_{\beta,\gamma\in\lambda}A_{\beta,\gamma}$ in    $\left(B^0_{\lambda}(S),\tau_{B}\right)$;

    \item[$(iii)$] If a non-empty subset $A\not\ni 0_S$ of $S$ is closed in $S$, then so is $A_{\beta,\gamma}$ in $\left(B^0_{\lambda}(S),\tau_{B}\right)$ for any    $\beta,\gamma\in\lambda$;

    \item[$(iv)$] If a non-empty subset $A\ni 0_S$ of $S$ is closed in $S$, then so is $\bigcup_{\beta,\gamma\in\lambda}A_{\beta,\gamma}$ in     $\left(B^0_{\lambda}(S),\tau_{B}\right)$;

    \item[$(v)$] If $x$ is a non-zero element of $S$ and $\mathscr{B}_x$ is a base of the topology $\tau$ at $x$, then the family $\mathscr{B}_{(\beta,x,\gamma)}=    \{U_{\beta,\gamma}\mid U\in\mathscr{B}_x\}$ is a base of the topology $\tau_B$ at the point $(\beta,x,\gamma)\in B^0_{\lambda}(S)$ for any $\beta,\gamma\in\lambda$;

    \item[$(vi)$] If $\mathscr{B}_{0_S}$ is a base of the topology $\tau$ at zero $0_S$ of $S$, then the family $\mathscr{B}_0=\{\bigcup_{\beta,\gamma\in\lambda} U_{\beta,\gamma}\mid U\in\mathscr{B}_{0_S}\}$ is a base of the topology $\tau_B$ at zero $0$ of the semigroup $B^0_{\lambda}(S)$.
\end{itemize}
\end{proposition}

\begin{proof}
$(i)$ Let $W\not\ni 0$ be an open set in $\left(B^0_{\lambda}(S),\tau_{B}\right)$ such that $W\cap S_{\alpha,\alpha}\in\tau_{B}|_{S_{\alpha,\alpha}}$. Then by Lemma~\ref{lemma-2.2}, $W\cap S_{\alpha,\alpha}^*$ is an open subset in $\left(B^0_{\lambda}(S),\tau_{B}\right)$.

By Definition~\ref{definition-1.1} the set $A_{\alpha,\alpha}$ is open for some $\alpha\in \lambda$. Since the map $\varphi^{(\alpha,\gamma)}_{(\delta,\alpha)}\colon B^0_{\lambda}(S)
\rightarrow B^0_{\lambda}(S)$ defined by the formula
$\varphi^{(\alpha,\gamma)}_{(\delta,\alpha)}(x)=(\alpha,1_S,\gamma)\cdot
x\cdot (\delta,1_S,\alpha)$ is continuous, the set $A_{\gamma,\delta}=
\big(\varphi^{(\alpha,\gamma)}_{(\delta,\alpha)}\big)^{-1}(A_{\alpha,\alpha})$ is open in $\left(B^0_{\lambda}(S),\tau_{B}\right)$ for any $\gamma,\delta\in\lambda$ as the full preimage of an open set under a continuous map.

$(ii)$ Let $A\ni 0_S$ be an open subset in $S$ and $W$ be an open subset
in $\left(B^0_{\lambda}(S),\tau_{B}\right)$ such that $W\cap
S_{\alpha,\alpha}=A_{\alpha,\alpha}$ for some $\alpha\in \lambda$.
Since the map $\varphi^{(\alpha,\gamma)}_{(\delta,\alpha)}\colon
B^0_{\lambda}(S) \rightarrow B^0_{\lambda}(S)$ is continuous for
any $\alpha,\gamma,\delta\in\lambda$ we conclude that
\begin{equation*}
    \widetilde{A}_{\gamma,\delta}=A_{\gamma,\delta}\cup
    \bigcup\left\{ S_{\iota,\kappa}\mid {(\iota,\kappa)\in(\lambda\times\lambda)\setminus \{(\gamma,\delta)\}}\right\}=
    \left(\varphi^{(\alpha,\gamma)}_{(\delta,\alpha)}\right)^{-1}(A_{\alpha,\alpha})
\end{equation*}
is an open subset in $\left(B^0_{\lambda}(S),\tau_{B}\right)$. Then since the cardinal $\lambda$ is finite, we have that $\displaystyle\bigcup_{\alpha,\beta\in\lambda} {A_{\alpha,\beta}}=\bigcap_{\gamma,\delta\in\lambda}\widetilde{A}_{\gamma,\delta}$ and this implies statement $(ii)$.

Statements $(iii)-(vi)$ follow from $(i)$ and $(ii)$.
\end{proof}

\begin{remark}\label{remark-2.5}
Note that the statements $(i)$, $(iii)$, $(iv)$ and $(v)$ of Proposition~\ref{proposition-2.4} hold for any infinite cardinal $\lambda$ and their proofs are similar to corresponding statements of Proposition~\ref{proposition-2.4} and follows from Lemma~\ref{lemma-2.2}. However, Example~2 and Lemma~7 of \cite{GutikPavlyk2005} imply that the statements $(ii)$ and $(vi)$ are false for any infinite cardinal $\lambda$ in the case of topological Brandt $\lambda^0$-extensions of topological semigroups in the class of topological semigroups (and hence in the case of topological Brandt $\lambda^0$-extensions of semitopological semigroups in the class of semitopological semigroups).
\end{remark}

Proposition~\ref{proposition-2.4} implies the following theorem which describes the structure of topological Brandt $\lambda^0$-extensions of semitopological monoids with zero in the class of semitopological semigroups for an arbitrary finite cardinal $\lambda\geqslant 1$.

\begin{theorem}\label{theorem-2.6}
For any semitopological monoid $(S,\tau)$ with zero and for any
finite cardinal $\lambda\geqslant 1$  there exists a unique
topological Brandt $\lambda^0$-extension
$\left(B^0_{\lambda}(S),\tau_{B}\right)$ of $(S,\tau)$ in the
class of semitopological semigroups, and the topology $\tau_{B}$
is generated by the base
$\mathscr{B}_B=\bigcup\left\{\mathscr{B}_B(t)\mid t\in
B^0_{\lambda}(S)\right\}$, where:
\begin{itemize}
    \item[$(i)$] $\mathscr{B}_B(t)=\big\{(U(s))_{\alpha,\beta} \setminus\{ 0_S\}\mid U(s)\in \mathscr{B}_S(s)\big\}$, where $t=(\alpha,s,\beta)$ is a non-zero element of    $B^0_{\lambda}(S)$, $\alpha,\beta\in\lambda$;

    \item[$(ii)$] $\mathscr{B}_B(0)=\left\{\bigcup_{\alpha,\beta\in\lambda}    (U(0_S))_{\alpha,\beta} \mid U(0_S)\in\mathscr{B}_S(0_S)\right\}$, where $0$ is the zero of $B^0_{\lambda}(S)$,
\end{itemize}
and $\mathscr{B}_S(s)$ is a base of the topology $\tau$ at the point $s\in S$.
\end{theorem}

A topological Brandt $\lambda^0$-extension $\left(B^0_{\lambda}(S), \tau_{B}\right)$ of a semitopological semigroup $(S, \tau)$ is called \emph{pseudocompact} (resp., \emph{countably pracompact}, \emph{countably compact}, \emph{sequentially compact}, \emph{compact}, \emph{semiregular}) if its underlying topological space is pseudocompact (resp., countably pracompact, countably compact, sequentially compact, compact, semiregular).

\begin{lemma}\label{lemma-2.6a}
A continuous image of a countably pracompact topological space is a countably pracompact space.
\end{lemma}

\begin{proof}
Let $X$ be a countably pracompact topological space and $f\colon
X\rightarrow Y$ be a continuous map from $X$ into a Hausdorff
topological space $Y$. Without loss of generality we may assume
that $f(X)=Y$. The countable pracompactness of $X$ implies that
there exists a dense subset $A$ in $X$ such that $X$ is countably
pracompact at $A$. Then we have that $f(A)$ is a dense subset of
$Y$ (see the proof of Theorem~1.4.10 from \cite{Engelking1989}).
We fix an arbitrary infinite subset $B\subseteq f(A)$. For every
$b\in B$ we choose an arbitrary point $a\in A$ such that
$f(a)=b$ and denote by $A^*$ the set of all points chosen in such
way. Then $A^*$ is infinite subset of $A$ and hence $A^*$ has an
accumulation point $x$ in $X$.   We claim that $f(x)$ is an
accumulation point of $B$ in $Y$. Otherwise there would exists an
open neighbourhood $U(f(x))$ of $f(x)$ in $Y$ such that $B\cap
U(f(x))=\varnothing$. Then the continuity of the map $f\colon
X\rightarrow Y$ implies that $V=f^{-1}\left(U(f(x))\right)$ is an
open neighbourhood of the point $x$ in $X$ such that
$f^{-1}(B)\cap V=\varnothing$, a contradiction.
\end{proof}

\begin{lemma}\label{lemma-2.7}
Let $\lambda$ be any cardinal $\geqslant 1$. If a topological Brandt $\lambda^0$-extension $\left(B^0_{\lambda}(S), \tau^*_{B}\right)$ of a semitopological monoid $(S, \tau)$ with zero in the class of Hausdorff semitopological semigroups is pseudocompact (resp., countably pracompact, countably compact, sequentially compact, compact), then the topological space $(S, \tau)$ is pseudocompact (resp., countably pracompact, countably compact, sequentially compact, compact).
\end{lemma}

\begin{proof}
We observe that any continuous image of a compact (resp., sequentially compact, countably compact, countably pracompact, pseudocompact) space is again a compact (resp., sequentially compact, countably compact, countably pracompact, pseudocompact) space (see \cite[Chapter~3]{Engelking1989} and Lemma~\ref{lemma-2.6a}). Let $\alpha\in\lambda$ be such that $\tau^*_{B}|_{S_{\alpha,\alpha}}=\tau$. Since $\left(B^0_{\lambda}(S), \tau^*_{B}\right)$ is a semitopological semigroup we have that the map $\varphi^{(\alpha,\alpha)}_{(\alpha,\alpha)}\colon B^0_{\lambda}(S)\rightarrow B^0_{\lambda}(S)$ defined by the formula $\varphi^{(\alpha,\alpha)}_{(\alpha,\alpha)}(x)=(\alpha,1_S,\alpha)\cdot x\cdot(\alpha,1_S,\alpha)$ is continuous and $\varphi^{(\alpha,\alpha)}_{(\alpha,\alpha)}\left(B^0_{\lambda}(S)\right)=S_{\alpha,\alpha}$. Then Definition~\ref{definition-1.1} completes the proof of the lemma.
\end{proof}

Next proposition will prove useful for constructing topology at zero of Brandt $\lambda^0$-extension for an infinite cardinal $\lambda$ in the class of countably compact semitopological semigroups.

\begin{proposition}\label{proposition-2.7a}
Let $\lambda$ be any cardinal $\geqslant 1$. Let
$\left(B^0_{\lambda}(S), \tau^*_{B}\right)$ be a topological
Brandt  $\lambda^0$-ex\-ten\-sion of a semitopological monoid $(S,
\tau)$ with zero in the class of semitopological semigroups. If
$\left(B^0_{\lambda}(S), \tau^*_{B}\right)$ is a Hausdorff
countably compact space then for every open neighbourhood $U(0)$
of the zero in $\left(B^0_{\lambda}(S), \tau^*_{B}\right)$ there
exist finitely many pairs of indices $(\alpha,\beta)$ such that
$S_{\alpha,\beta}\nsubseteq U(0)$.
\end{proposition}

\begin{proof}
Suppose to the contrary that there exists an open neighbourhood
$U(0)$  of the zero $0$ in $\left(B^0_{\lambda}(S),
\tau^*_{B}\right)$ such that $S_{\alpha,\beta}\nsubseteq U(0)$ for
infinitely many pairs of indices $(\alpha,\beta)$. Then for every
such $S_{\alpha,\beta}$ we fix a unique point $x_{\alpha,\beta}\in
S_{\alpha,\beta}\setminus U(0)$ and put
$A=\bigcup\{x_{\alpha,\beta}\}$. Then $A$ is infinite and
Lemma~\ref{lemma-2.2} implies that the set $A$ has no accumulation
point in $\left(B^0_{\lambda}(S), \tau^*_{B}\right)$. This
contradicts Theorem~3.10.3 of \cite{Engelking1989}. The obtained
contradiction implies the statement of the proposition.
\end{proof}

Proposition~\ref{proposition-2.7a} implies the following corollary:

\begin{corollary}\label{corollary-2.7b}
Let $\lambda$ be any cardinal $\geqslant 1$. Let
$\left(B^0_{\lambda}(S), \tau^*_{B}\right)$  be a topological
Brandt $\lambda^0$-extension of a semitopological monoid $(S,
\tau)$ with zero in the class of semitopological semigroups. If
$\left(B^0_{\lambda}(S), \tau^*_{B}\right)$ is a Hausdorff
sequentially compact (compact) space then for
every open neighbourhood $U(0)$ of the zero in
$\left(B^0_{\lambda}(S), \tau^*_{B}\right)$ there exist finitely
many pairs of indices $(\alpha,\beta)$ such that
$S_{\alpha,\beta}\nsubseteq U(0)$.
\end{corollary}

The following theorem describes the structure of Hausdorff countably compact topological Brandt $\lambda^0$-exten\-si\-ons of semitopological monoids with zero in the class of semitopological semigroups.

\begin{theorem}\label{theorem-2.7c}
For any Hausdorff countably compact  semitopological  monoid
$(S,\tau)$ with zero and for any cardinal $\lambda\geqslant 1$
there exists a unique Hausdorff countably compact  topological
Brandt $\lambda^0$-extension
$\left(B^0_{\lambda}(S),\tau_{B}^S\right)$ of $(S,\tau)$ in the
class of semitopological semigroups, and the topology $\tau_{B}^S$
is generated by the base
$\mathscr{B}_B=\bigcup\left\{\mathscr{B}_B(t)\mid t\in
B^0_{\lambda}(S)\right\}$, where:
\begin{itemize}
    \item[$(i)$] $\mathscr{B}_B(t)=\big\{(U(s))_{\alpha,\beta} \setminus\{ 0_S\}\mid U(s)\in \mathscr{B}_S(s)\big\}$, where $t=(\alpha,s,\beta)$ is a non-zero element of    $B^0_{\lambda}(S)$, $\alpha,\beta\in\lambda$;

    \item[$(ii)$] $\mathscr{B}_B(0)=\Big\{U_{A}(0)= \bigcup_{(\alpha,\beta)\in(\lambda\times\lambda) \setminus A}S_{\alpha,\beta}\cup \bigcup_{(\gamma,\delta)\in A} (U(0_S))_{\gamma,\delta} \mid A \hbox{~is a finite subset of~} \lambda\times\lambda$ and  $U(0_S)\in\mathscr{B}_S(0_S)\Big\}$, where $0$ is the zero of $B^0_{\lambda}(S)$,
\end{itemize}
and $\mathscr{B}_S(s)$ is a base of the topology $\tau$ at the point $s\in S$.
\end{theorem}

\begin{proof}
In the case when $\lambda$ is a finite cardinal, Theorem~\ref{theorem-2.6} and the definition of a countably compact space imply the statements of the theorem. Also, in this case the proof of the separate continuity of the semigroup operation in $\left(B^0_{\lambda}(S),\tau_{B}^S\right)$ is trivial, and hence we omit it.

Suppose that $\lambda$ is an infinite cardinal. Then statement $(i)$ follows from Remark~\ref{remark-2.5} and Proposition~\ref{proposition-2.7a} implies that $\mathscr{B}_B(0)$ is a base of the topology $\tau_{B}^S$ at the zero $0\in B^0_{\lambda}(S)$. The proof is completed by showing that the semigroup operation in $\left(B^0_{\lambda}(S),\tau_{B}^S\right)$ is separately continuous. We consider only the cases $0\cdot (\alpha,s,\beta)$ and $ (\alpha,s,\beta)\cdot 0$, because in other cases the proof of the separate continuity of the semigroup operation in $\left(B^0_{\lambda}(S),\tau_{B}^S\right)$ is trivial.

Let $K$ be an arbitrary finite subset in $\lambda\times\lambda$
and let $U(0_S)$  be an arbitrary open neighbourhood of the zero
$0_S$ in $(S,\tau)$. Then there exists an open neighbourhood
$V(0_S)$ in $(S,\tau)$ such that $s\cdot V(0_S)\subseteq U(0_S)$
and $V(0_S)\cdot s\subseteq U(0_S)$. Let $K_1$ be such subset of
$\lambda$ that $K\cup\{(\alpha,\beta), (\beta,\alpha)\}\subseteq
K_1\times K_1$. Then we have that
\begin{equation*}
    V_{K_1\times K_1}(0)\cdot (\alpha,s,\beta)\subseteq U_K(0) \qquad \hbox{and} \qquad
    (\alpha,s,\beta)\cdot V_{K_1\times K_1}(0)\subseteq U_K(0).
\end{equation*}
This completes the proof of the theorem.
\end{proof}

Theorem~\ref{theorem-2.7c} implies the following corollary:

\begin{corollary}\label{corollary-2.7d}
For any Hausdorff sequentially compact (resp., compact)
semitopological monoid $(S,\tau)$ with zero and for any cardinal
$\lambda\geqslant 1$ there exists a unique Hausdorff sequentially
compact (resp., compact) topological Brandt $\lambda^0$-extension
$\left(B^0_{\lambda}(S),\tau_{B}^S\right)$ of $(S,\tau)$ in the
class of semitopological semigroups, and the topology $\tau_{B}^S$
is generated by the base
$\mathscr{B}_B=\bigcup\left\{\mathscr{B}_B(t)\mid t\in
B^0_{\lambda}(S)\right\}$, where:
\begin{itemize}
    \item[$(i)$] $\mathscr{B}_B(t)=\big\{(U(s))_{\alpha,\beta} \setminus\{ 0_S\}\mid U(s)\in \mathscr{B}_S(s)\big\}$, where $t=(\alpha,s,\beta)$ is a non-zero element of    $B^0_{\lambda}(S)$, $\alpha,\beta\in\lambda$;

    \item[$(ii)$]  $\mathscr{B}_B(0)=\Big\{U_{A}(0)= \bigcup_{(\alpha,\beta)\in(\lambda\times\lambda) \setminus A}S_{\alpha,\beta}\cup \bigcup_{(\gamma,\delta)\in A} (U(0_S))_{\gamma,\delta} \mid A \hbox{~is a finite subset of~} \lambda\times\lambda$ and  $U(0_S)\in\mathscr{B}_S(0_S)\Big\}$, where $0$ is the zero of $B^0_{\lambda}(S)$,
\end{itemize}
and $\mathscr{B}_S(s)$ is a base of the topology $\tau$ at the point $s\in S$.
\end{corollary}

We proceed to show that the statement analogous to Proposition~\ref{proposition-2.7a} holds in case $B^0_{\lambda}(S)$ is a semiregular pseudocompact space.

\begin{proposition}\label{proposition-2.8}
Let $\lambda$ be any cardinal $\geqslant 1$. Let
$\left(B^0_{\lambda}(S), \tau^*_{B}\right)$ be a topological
Brandt $\lambda^0$-ex\-ten\-sion of a semitopological monoid $(S,
\tau)$ with zero in the class of semitopological semigroups. If
$\left(B^0_{\lambda}(S), \tau^*_{B}\right)$ is a semiregular
pseudocompact space then for every open neighbourhood $U(0)$ of
the zero in $\left(B^0_{\lambda}(S), \tau^*_{B}\right)$ there
exist finitely many pairs of indices $(\alpha,\beta)$ such that
$S_{\alpha,\beta}\nsubseteq U(0)$.
\end{proposition}

\begin{proof}
Let $\mathscr{B}(0)$ be a base of the topology $\tau^*_{B}$ on
$B^0_{\lambda}(S)$ which consists of  regular open subsets of
$\left(B^0_{\lambda}(S), \tau^*_{B}\right)$. Fix an arbitrary
$U(0)\in\mathscr{B}(0)$. We claim that there exist finitely many
pairs of indices $(\alpha,\beta)$ such that
$S_{\alpha,\beta}\nsubseteq
\operatorname{cl}_{B^0_{\lambda}(S)}(U(0))$. Otherwise, by
Lemma~\ref{lemma-2.2} we have that
$\mathscr{P}=\left\{S_{\alpha,\beta}^*\setminus
\operatorname{cl}_{B^0_{\lambda}(S)}(U(0))\mid
\alpha,\beta\in\lambda\right\}$ is a family of open subsets of the
topological space $\left(B^0_{\lambda}(S), \tau^*_{B}\right)$.
Simple verifications show that $\mathscr{P}$ is an infinite
locally finite family, which contradicts the pseudocompactness of
the space $\left(B^0_{\lambda}(S), \tau^*_{B}\right)$. The
obtained contradiction implies that there exist finitely many pais
of indices $(\alpha,\beta)$ such that $S_{\alpha,\beta}\nsubseteq
\operatorname{cl}_{B^0_{\lambda}(S)}(U(0))$.

Let $V(0)$ be any open neighbourhood of the zero $0$ in
$\left(B^0_{\lambda}(S), \tau^*_{B}\right)$.  Then there exists an
element $U(0)$ of the base $\mathscr{B}(0)$ of the topology
$\tau^*_{B}$ on $B^0_{\lambda}(S)$ which consists of regular open
subsets of $\left(B^0_{\lambda}(S), \tau^*_{B}\right)$ such that
$V(0)\subseteq U(0)$ and by the above there exist finitely many pairs of
indices $(\alpha,\beta)$ such that $S_{\alpha,\beta}\nsubseteq
\operatorname{cl}_{B^0_{\lambda}(S)}(U(0))$. Consider an arbitrary set
$S_{\alpha,\beta}$ such that $S_{\alpha,\beta}\subseteq
\operatorname{cl}_{B^0_{\lambda}(S)}(U(0))$. Then by
Lemma~\ref{lemma-2.2} we have that $S^*_{\alpha,\beta}\subseteq
\operatorname{int}_{B^0_{\lambda}(S)}
(\operatorname{cl}_{B^0_{\lambda}(S)}(U(0)))$ and since $U(0)$ is
a regular open set we conclude that $S_{\alpha,\beta}\subseteq
U(0)$. This completes the proof of the proposition.
\end{proof}

The following theorem describes the structure of semiregular pseudocompact topological Brandt $\lambda^0$-exten\-si\-ons of semitopological monoids with zero in the class of semitopological semigroups.

\begin{theorem}\label{theorem-2.10}
For any semiregular pseudocompact semitopological monoid $(S,\tau)$ with zero and for any cardinal $\lambda\geqslant 1$ there exists a unique semiregular pseudocompact topological Brandt $\lambda^0$-extension $\left(B^0_{\lambda}(S),\tau_{B}^S\right)$ of $(S,\tau)$ in the class of semitopological semigroups, and the topology $\tau_{B}^S$ is generated by the base $\mathscr{B}_B=\bigcup\left\{\mathscr{B}_B(t)\mid t\in B^0_{\lambda}(S)\right\}$, where:
\begin{itemize}
    \item[$(i)$] $\mathscr{B}_B(t)=\big\{(U(s))_{\alpha,\beta} \setminus\{ 0_S\}\mid U(s)\in \mathscr{B}_S(s)\big\}$, where $t=(\alpha,s,\beta)$ is a non-zero element of    $B^0_{\lambda}(S)$, $\alpha,\beta\in\lambda$;

    \item[$(ii)$] $\mathscr{B}_B(0)=\Big\{U_{A}(0)= \bigcup_{(\alpha,\beta)\in(\lambda\times\lambda) \setminus A}S_{\alpha,\beta}\cup \bigcup_{(\gamma,\delta)\in A} (U(0_S))_{\gamma,\delta} \mid A \hbox{~is a finite subset of~} \lambda\times\lambda$ and  $U(0_S)\in\mathscr{B}_S(0_S)\Big\}$, where $0$ is the zero of $B^0_{\lambda}(S)$,
\end{itemize}
and $\mathscr{B}_S(s)$ is a base of the topology $\tau$ at the point $s\in S$.
\end{theorem}

The proof of Theorem~\ref{theorem-2.10} is similar to the proof of Theorem~\ref{theorem-2.7c} and is based on Proposition~\ref{proposition-2.8}.

Theorem~\ref{theorem-2.10} implies the following corollary.

\begin{corollary}\label{corollary-2.11}
For any semiregular countably pracompact semitopological monoid $(S,\tau)$ with zero and for any cardinal $\lambda\geqslant 1$ there exists a unique semiregular countably pracompact topological Brandt $\lambda^0$-extension $\left(B^0_{\lambda}(S),\tau_{B}^S\right)$ of $(S,\tau)$ in the class of semitopological semigroups, and the topology $\tau_{B}^S$ is generated by the base $\mathscr{B}_B=\bigcup\left\{\mathscr{B}_B(t)\mid t\in B^0_{\lambda}(S)\right\}$, where:
\begin{itemize}
    \item[$(i)$] $\mathscr{B}_B(t)=\big\{(U(s))_{\alpha,\beta} \setminus\{ 0_S\}\mid U(s)\in \mathscr{B}_S(s)\big\}$, where $t=(\alpha,s,\beta)$ is a non-zero element of    $B^0_{\lambda}(S)$, $\alpha,\beta\in\lambda$;

    \item[$(ii)$] $\mathscr{B}_B(0)=\Big\{U_{A}(0)= \bigcup_{(\alpha,\beta)\in(\lambda\times\lambda) \setminus A}S_{\alpha,\beta}\cup \bigcup_{(\gamma,\delta)\in A} (U(0_S))_{\gamma,\delta} \mid A \hbox{~is a finite subset of~} \lambda\times\lambda$ and  $U(0_S)\in\mathscr{B}_S(0_S)\Big\}$, where $0$ is the zero of $B^0_{\lambda}(S)$,
\end{itemize}
and $\mathscr{B}_S(s)$ is a base of the topology $\tau$ at the point $s\in S$.
\end{corollary}

By Theorems~\ref{theorem-2.7c}, \ref{theorem-2.10} and Corollaries~\ref{corollary-2.7d}, \ref{corollary-2.11}
we have that for every semiregular pseudocompact (resp.,
semiregular countably pracompact, Hausdorff countably compact,
Hausdorff sequentially compact, Hausdorff compact) semitopological
semigroup $(S,\tau)$ there exists a unique semiregular
pseudocompact (resp., semiregular countably pracompact, Hausdorff
countably compact, Hausdorff sequentially compact, Hausdorff
compact) topological Brandt $\lambda^0$-extension
$\left(B^0_{\lambda}(S),\tau_{B}^S\right)$ of $(S,\tau)$ in the
class of semitopological semigroups. This implies the following
corollary, which generalized Theorem~2 from
\cite{GutikPavlyk2005}.

\begin{corollary}\label{corollary-2.12}
Let $\lambda$ be an arbitrary cardinal $\geqslant 1$ and $(S,\tau)$ be a Hausdorff semitopological monoid with zero. Then the following assertions hold:
\begin{itemize}
  \item[$(i)$] if $(S,\tau)$ is compact then the semiregular pseudocompact (resp., semiregular countably pracompact, Hausdorff countably compact) topological Brandt $\lambda^0$-extension $\left(B^0_{\lambda}(S),\tau_{B}^S\right)$ of $(S,\tau)$ in the class of semitopological semigroups is compact;
  \item[$(ii)$] if $(S,\tau)$ is semiregular sequentially compact then the semiregular pseudocompact (resp., semiregular countably pracompact, Hausdorff countably compact) topological Brandt $\lambda^0$-extension $\left(B^0_{\lambda}(S),\tau_{B}^S\right)$ of $(S,\tau)$ in the class of semitopological semigroups is sequentially compact;
  \item[$(iii)$] if $(S,\tau)$ is semiregular countably compact then the semiregular pseudocompact (resp., semiregular countably pracompact) topological Brandt $\lambda^0$-extension $\left(B^0_{\lambda}(S),\tau_{B}^S\right)$ of $(S,\tau)$ in the class of semitopological semigroups is countably compact;
  \item[$(iv)$] if $(S,\tau)$ is semiregular countably pracompact then the semiregular pseudocompact topological Brandt $\lambda^0$-extension $\left(B^0_{\lambda}(S),\tau_{B}^S\right)$ of $(S,\tau)$ in the class of semitopological semigroups is countably pracompact.
\end{itemize}
\end{corollary}

The following example shows that the statements of Theorem~\ref{theorem-2.10}, Corollaries~\ref{corollary-2.11} and \ref{corollary-2.12} are not true in the case when the pseudocompact (resp., countably pracompact) space $\left(B^0_{\lambda}(S),\tau_{B}^S\right)$ is not semiregular.

\begin{example}\label{example-2.12a}
Let $\lambda$ be any infinite cardinal. Let $S$ be the set
$\{1\}\cup\left\{1-\frac{1}{n}\colon n\in\omega\right\}$ with the
semilattice operation $\min$ and the usual topology $\tau_u$.
We define a topology $\tau_B^\circ$ on $B^0_{\lambda}(S)$ in the
following way:
\begin{itemize}
  \item[$(i)$] the family $\mathscr{B}_B^\circ(\alpha,s,\beta)=\{(\alpha,U,\beta)\mid U\in\tau_u\}$ is the base of the topology $\tau_B^\circ$ at the non-zero element $(\alpha,s,\beta)\in B^0_{\lambda}(S)$; \; and
  \item[$(ii)$] the family $\mathscr{B}_B^\circ(0)=\{U_{(\alpha_1,\beta_1), \ldots,(\alpha_n,\beta_n)}\mid n\in\omega\}$, where
\begin{equation*}
    U_{(\alpha_1,\beta_1), \ldots,(\alpha_n,\beta_n)}= \bigcup\left\{ S_{\alpha,\beta}\setminus\{(\alpha,1,\beta)\}\mid (\alpha,\beta)\in(\lambda\times\lambda)\setminus\{(\alpha_1,\beta_1), \ldots,(\alpha_n,\beta_n)\}\right\},
\end{equation*}
      is the base of the topology $\tau_B^\circ$ at the zero $0\in B^0_{\lambda}(S)$.
\end{itemize}

Simple verifications show that $\left(B^0_{\lambda}(S),\tau_B^\circ\right)$ is a Hausdorff non-semiregular topological space. Since every nonzero element $(\alpha,s,\beta)$ of the topological space $\left(B^0_{\lambda}(S),\tau_B^\circ\right)$ has a base which consists of open-and-closed subsets we conclude that the topological space $\left(B^0_{\lambda}(S),\tau_B^\circ\right)$ is functionally Hausdorff (i.e., for every two distinct points $x,y\in X$ there exists a continuous function $f\colon X\rightarrow[0,1]$ such that $f(x)=0$ and $f(y)=1$).
\end{example}

\begin{proposition}\label{proposition-2.12b}
If $\lambda$ is an infinite cardinal then
$\left(B^0_{\lambda}(S),\tau_B^\circ\right)$ is  a countably
pracompact semitopological semigroup, where $S$ is the topological semigroup defined in Example~\ref{example-2.12a}.
\end{proposition}

\begin{proof}
First we show that the topological space $\left(B^0_{\lambda}(S),\tau_B^\circ\right)$ is countably pracompact. Simple observations show that the set $A=\left\{(\alpha,s,\beta)\mid s\in S\setminus\{0,1\}\right\}$ is dense in $\left(B^0_{\lambda}(S),\tau_B^\circ\right)$. Let $B$ be an infinite subset of $A$. If for some $S_{\alpha,\beta}$ the set $B\cap S_{\alpha,\beta}$ is infinite then the definition of the topology $\tau_B^\circ$ implies that $(\alpha,1,\beta)$ is an accumulation point of $B$ in $\left(B^0_{\lambda}(S),\tau_B^\circ\right)$. Otherwise, the definition of the topology $\tau_B^\circ$ implies that zero $0$ is an accumulation point of $B$ in $\left(B^0_{\lambda}(S),\tau_B^\circ\right)$.

The task is now to show that $\left(B^0_{\lambda}(S),\tau_B^\circ\right)$ is a semitopological semigroup. For every \break $U_{(\alpha_1,\beta_1), \ldots,(\alpha_n,\beta_n)}\in\mathscr{B}_B^\circ(0)$ and any $(\alpha,s,\beta), (\gamma,t,\delta)\in B^0_{\lambda}(S)$ the following conditions hold:
\begin{itemize}
  \item[$(i)$] if $\beta=\gamma$ then $(\alpha,V(s),\beta)\cdot (\gamma,V(t),\delta) \subseteq (\alpha,U(s\cdot t),\delta)$ for open neighbourhoods $V(s)$, $V(t)$ and $U(s\cdot t)$ of $s$, $t$ and $s\cdot t$ in $S$, respectively, such that $V(s)\cdot V(t)\subseteq U(s\cdot t)$, because $S$ is a topological semigroup;
      
  \item[$(ii)$] if $\beta\neq\gamma$ then $(\alpha,V(s),\beta)\cdot (\gamma,V(t),\delta)=\{0\} \subseteq U_{(\alpha_1,\beta_1), \ldots,(\alpha_n,\beta_n)}$ for all open neighbourhoods $V(s)$ and $V(t)$ of $s$ and $t$ in $S$, respectively;
      
  \item[$(iii)$] $(\alpha,V(s),\beta)\cdot \{0\}=\{0\}\cdot(\alpha,V(s),\beta) =\{0\} \subseteq U_{(\alpha_1,\beta_1), \ldots,(\alpha_n,\beta_n)}$ for every open neighbourhood $V(s)$ of $s$ in $S$;
      
  \item[$(iv)$] $U_{(\alpha_1,\beta_1), \ldots,(\alpha_n,\beta_n)}\cdot \{0\}=\{0\}\cdot U_{(\alpha_1,\beta_1), \ldots,(\alpha_n,\beta_n)}=\{0\} \subseteq U_{(\alpha_1,\beta_1), \ldots,(\alpha_n,\beta_n)}$;
      
  \item[$(v)$] $(\alpha,s,\beta)\cdot U_{(\beta,\beta),(\alpha_1,\beta_1), \ldots,(\alpha_n,\beta_n)(\beta_1,\beta_1), \ldots,(\beta_n,\beta_n)} \subseteq U_{(\alpha_1,\beta_1), \ldots,(\alpha_n,\beta_n)}$;
      
  \item[$(vi)$] $U_{(\alpha,\alpha),(\alpha_1,\beta_1), \ldots,(\alpha_n,\beta_n)(\alpha_1,\alpha_1), \ldots,(\alpha_n,\alpha_n)} \cdot(\alpha,s,\beta) \subseteq U_{(\alpha_1,\beta_1), \ldots,(\alpha_n,\beta_n)}$.
\end{itemize}
Hence the semigroup operation in $\left(B^0_{\lambda}(S),\tau_B^\circ\right)$ is separately continuous. This completes the proof of the proposition.
\end{proof}

\begin{remark}\label{remark-2.12c}
In the case when $\lambda$ is an infinite countable cardinal  
$\left(B^0_{\lambda}(S),\tau_B^\circ\right)$ is not a metrizable
space because $\left(B^0_{\lambda}(S),\tau_B^\circ\right)$ is a second countable space which is not regular (see Theorem~4.2.9 of~\cite{Engelking1989}).
\end{remark}

Here and subsequently $\tau_{B}^S$ denotes the the topology on the Brandt $\lambda^0$-extension $B^0_{\lambda}(S)$ of a semitopological semigroup $S$ which is defined in Theorems~\ref{theorem-2.7c} and \ref{theorem-2.10}.

The following proposition is proved by detailed verifications.

\begin{proposition}\label{proposition-2.18}
Let $\lambda$ by any cardinal $\geqslant 1$. Then $\left(B^0_{\lambda}(S),\tau_{B}^{S}\right)$ is a Hausdorff semitopological semigroup if and only if $(S,\tau)$ is a Hausdorff semitopological semigroup.
\end{proposition}

\begin{proposition}\label{proposition-2.19}
Let $\lambda\geqslant 1$ by any cardinal. Then the following conditions hold:
\begin{itemize}
  \item[$(i)$] the topological space $\left(B^0_{\lambda}(S),\tau_B^{S}\right)$  is regular if and only if the space $S$ is regular;

  \item[$(ii)$] the topological space $\left(B^0_{\lambda}(S),\tau_B^{S}\right)$ is Tychonoff if and only if the space $S$ is Tychonoff;

  \item[$(iii)$] the topological space $\left(B^0_{\lambda}(S),\tau_B^{S}\right)$ is normal if and only if the space $S$ is normal.
\end{itemize}
\end{proposition}

\begin{proof}
$(i)$ Implication $(\Rightarrow)$ follows from Theorem~2.1.6 of \cite{Engelking1989}.

$(\Leftarrow)$ Suppose that the space $S$ is regular. We consider
the case when cardinal $\lambda$ is infinite. In the case
$\lambda<\omega$ the proof is similar and follows from
Theorem~\ref{theorem-2.6}.

Fix an arbitrary point $x\in B^0_{\lambda}(S)$. Suppose that $x\in
S^*_{\alpha,\beta}$ for some $\alpha,\beta\in\lambda$. Then the
regularity of $S$, Remark~\ref{remark-2.5} and the definition of the topology $\tau_B^{S}$ on $B^0_{\lambda}(S)$ imply that
there exist open neighbourhoods $V(x)$ and $U(x)$ of $x$ in
$\left(B^0_{\lambda}(S),\tau_B^S\right)$ such that
$\operatorname{cl}_{B^0_{\lambda}(S)}(V(x))\subseteq U(x)\subseteq
S^*_{\alpha,\beta}$.

We now turn to the case $x=0$. Let $U_A(0)$ be a basic open neighbourhood of $0$ in $\left(B^0_{\lambda}(S),\tau_B^S\right)$. Now, the regularity of the space $S$ implies that there exists an open neighbourhood $V(0_S)$ of $0_S$ in $S$ such that
$\operatorname{cl}_{S}(V(0_S))\subseteq U(0_S)$. Next, by Remark~\ref{remark-2.5} the set $(S\setminus \operatorname{cl}_{S}(V(0_S)))_{\alpha,\beta}$ is open in $\left(B^0_{\lambda}(S),\tau_B^S\right)$ for any $\alpha,\beta\in\lambda$, and hence Remark~\ref{remark-2.5} implies the following inclusion $\operatorname{cl}_{B^0_{\lambda}(S)}(V_A(0))\subseteq U_A(0)$. This completes the proof of the statement.

$(ii)$ Implication $(\Rightarrow)$ follows from Theorem~2.1.6 of \cite{Engelking1989}.

$(\Leftarrow)$ Suppose that the space $S$ is Tychonoff. We consider the case when cardinal $\lambda$ is infinite. In the case $\lambda<\omega$ the proof is similar applying Theorem~\ref{theorem-2.6}.

Fix an arbitrary point $x\in B^0_{\lambda}(S)$. Suppose that $x\in
S^*_{\alpha,\beta}$ for some $\alpha,\beta\in\lambda$. Since by
Lem\-ma~\ref{lemma-2.2} the subspace $S_{\alpha,\beta}$ is
homeomorphic to $S$, we identify $S_{\alpha,\beta}$ with $S$. Let
$U(x)$ be an open neighbourhood of the point $x$ in
$\left(B^0_{\lambda}(S),\tau_B^S\right)$ such that $U(x)\subseteq
S_{\alpha,\beta}$. Then by Proposition~1.5.8 of
\cite{Engelking1989} there exists a continuous function $f\colon
S=S_{\alpha,\beta}\rightarrow[0,1]$ such that $f(x)=1$ and
$f(y)=0$ for all $y\in S_{\alpha,\beta}\setminus U(x)$. We define
the map $\varphi\colon B^0_{\lambda}(S)\rightarrow
B^0_{\lambda}(S)$ by the formula
$\varphi(x)=(\alpha,1_S,\alpha)\cdot x\cdot (\beta,1_S,\beta)$.
Since $\left(B^0_{\lambda}(S),\tau_B^S\right)$ is a semitopological
semigroup we conclude that such map $\varphi$ is
continuous and hence the map $F\colon
B^0_{\lambda}(S)\rightarrow[0,1]$ defined by the formula
$F(x)=f(\varphi(x))$ is continuous. Moreover we have that $F(x)=1$
and $F(y)=0$ for all $y\in B^0_{\lambda}(S)\setminus U(x)$.

It remains to prove that $\left(B^0_{\lambda}(S),\tau_B^S\right)$ is Tychonoff for $x=0$. Let $n$ be any positive integer. Then by
Proposition~1.5.8 of  \cite{Engelking1989} for every open
neighbourhood $V^i(0_S)$ of zero $0_S$ in $S$, $i=1,\ldots,n$,
there exists a continuous function $f_{V^i}\colon S\rightarrow
[0,1]$ such that $f_{V^i}(0_S)=1$ and $f_{V^i}(y)=0$ for all $y\in
S\setminus V^i(0_S)$. The definition of the topology $\tau_B^{S}$ on $B^0_{\lambda}(S)$ implies that for
every open neighbourhood $W(0)$ of zero $0$ in
$\left(B^0_{\lambda}(S),\tau_B^S\right)$ there exists finitely many
indices
$(\alpha_1,\beta_1),\ldots,(\alpha_n,\beta_n)\in\lambda\times\lambda$
and finitely many neighborhoods $V^1(0_S),\ldots,V^n(0_S)$ of zero
$0_S$ in $S$ such that
$W(0)=V^1_{(\alpha_1,\beta_1)}(0)\cap\cdots\cap
V^n_{(\alpha_n,\beta_n)}(0)$. For any $\alpha,\beta\in\lambda$ we
identify the topological space $S_{\alpha,\beta}$ with $S$. Let
$\varphi\colon B^0_{\lambda}(S)\rightarrow B^0_{\lambda}(S)$ be
the map defined in the proof of statement $(i)$. Then the map
$F_i\colon B^0_{\lambda}(S)\rightarrow[0,1]$ defined by the
formula $F_i(x)=f_{V^i}(\varphi(x))$ is continuous as a
composition of two continuous maps. Now the map $F\colon
B^0_{\lambda}(S)\rightarrow[0,1]$ defined by the formula
$F(x)=\min\{F_1(x),\ldots,F_n(x)\}$ becomes continuous (see
\cite[Section~1.4]{Engelking1989}). Simple verifications show that
$F(0)=1$ and $F(y)=0$ for all $y\in B^0_{\lambda}(S)\setminus
W(0)$. Then Proposition~1.5.8 of \cite{Engelking1989} completes
the proof of the statement.

Next we shall prove statement $(iii)$.

$(\Rightarrow)$ Suppose that $\left(B^0_{\lambda}(S),\tau_B^{S}\right)$ is a normal space. By Lemma~\ref{lemma-2.2} we have that $S_{\alpha,\alpha}$ is a closed subset of $\left(B^0_{\lambda}(S),\tau_B^{S}\right)$ for every $\alpha\in\lambda$. Then by Theorem~2.1.6 from \cite{Engelking1989},  $S_{\alpha,\alpha}$ is a normal subspace of $\left(B^0_{\lambda}(S),\tau_B^{S}\right)$ and hence Definition~\ref{definition-1.1} and Lemma~\ref{lemma-2.1} imply that $S$ is a normal space.

$(\Leftarrow)$ Suppose that $S$ is a normal space. Let $F_1$ and $F_2$ be arbitrary closed disjoint subsets of $\left(B^0_{\lambda}(S),\tau_B^{S}\right)$. Then only one of the following cases holds:
\begin{itemize}
  \item[1)] $0\in F_1$ and $0\notin F_2$;
  \item[2)] $0\notin F_1$ and $0\notin F_2$.
\end{itemize}

Suppose case 1) holds. Since $0\notin F_2$ we conclude that the definition of the topology $\tau_B^{S}$ on $B^0_{\lambda}(S)$ implies that there exists finitely many pairs of indices $(\alpha_1,\beta_1),\ldots, (\alpha_n,\beta_n)$ such that $F_2\subseteq S^*_{\alpha_1,\beta_1}\cup\cdots\cup S^*_{\alpha_n,\beta_n}$. By Definition~\ref{definition-1.1} and Lemma~\ref{lemma-2.1} we have that the space $S$ is homeomorphic to the subspace $S_{\alpha,\beta}$ of $\left(B^0_{\lambda}(S),\tau_B^{S}\right)$ and hence $S_{\alpha,\beta}$ is a normal subspace of $\left(B^0_{\lambda}(S),\tau_B^{S}\right)$ for all $\alpha,\beta\in\lambda$. Then there exist finitely many pairs of open subsets $\left(U^1_{\alpha_1,\beta_1},U^2_{\alpha_1,\beta_1}\right),\ldots, \left(U^1_{\alpha_n,\beta_n},U^2_{\alpha_n,\beta_n}\right)$ in $S_{\alpha_1,\beta_1},\ldots,S_{\alpha_n,\beta_n}$, respectively, such that the following conditions hold:
\begin{align*}
     F_1\cap S_{\alpha_1,\beta_1}\subseteq U^1_{\alpha_1,\beta_1}\subseteq S_{\alpha_1,\beta_1},\;\ldots,\;
    & F_1\cap S_{\alpha_n,\beta_n}\subseteq U^1_{\alpha_n,\beta_n}\subseteq S_{\alpha_n,\beta_n},\\
     F_2\cap S_{\alpha_1,\beta_1}\subseteq U^2_{\alpha_1,\beta_1}\subseteq S_{\alpha_1,\beta_1},\;\ldots,\;
    & F_2\cap S_{\alpha_n,\beta_n}\subseteq U^2_{\alpha_n,\beta_n}\subseteq S_{\alpha_n,\beta_n}.
\end{align*}
Put
\begin{equation*}
    U(F_1)=U^1_{\alpha_1,\beta_1}\cup\cdots\cup U^1_{\alpha_n,\beta_n}\cup \bigcup\left\{S_{\alpha,\beta}\mid (\alpha,\beta)\in\lambda\times\lambda \setminus\{(\alpha_1,\beta_1),\ldots,(\alpha_n,\beta_n)\}\right\}
\end{equation*}
and $U(F_2)=U^2_{\alpha_1,\beta_1}\cup\cdots\cup U^2_{\alpha_n,\beta_n}$. Then we have that $U(F_1)\cap U(F_2)=\varnothing$ and the definition of the topology $\tau_B^{S}$ implies that $U(F_1)$ and $U(F_2)$ are open neighbourhoods of closed sets $F_1$ and $F_2$ in $\left(B^0_{\lambda}(S),\tau_B^{S}\right)$, respectively.

Suppose case 2) holds. Then the definition of the topology $\tau_B^{S}$ implies that there exists finitely many pairs of indices $(\alpha_1,\beta_1),\ldots, (\alpha_n,\beta_n)$ such that $F_1\cup F_2\subseteq S^*_{\alpha_1,\beta_1}\cup\cdots\cup S^*_{\alpha_n,\beta_n}$. Now, similarly as in case 1) we can find finitely many pairs of open subsets $\left(U^1_{\alpha_1,\beta_1},U^2_{\alpha_1,\beta_1}\right),\ldots,$ $\left(U^1_{\alpha_n,\beta_n},U^2_{\alpha_n,\beta_n}\right)$ in $S_{\alpha_1,\beta_1},\ldots,S_{\alpha_n,\beta_n}$, respectively, such that the following conditions hold:
\begin{align*}
     F_1\cap S_{\alpha_1,\beta_1}\subseteq U^1_{\alpha_1,\beta_1}\subseteq S_{\alpha_1,\beta_1},\;\ldots,\;
    & F_1\cap S_{\alpha_n,\beta_n}\subseteq U^1_{\alpha_n,\beta_n}\subseteq S_{\alpha_n,\beta_n},\\
     F_2\cap S_{\alpha_1,\beta_1}\subseteq U^2_{\alpha_1,\beta_1}\subseteq S_{\alpha_1,\beta_1},\;\ldots,\;
    & F_2\cap S_{\alpha_n,\beta_n}\subseteq U^2_{\alpha_n,\beta_n}\subseteq S_{\alpha_n,\beta_n}.
\end{align*}
We put $U(F_1)=U^1_{\alpha_1,\beta_1}\cup\cdots\cup U^1_{\alpha_n,\beta_n}$ and $U(F_2)=U^2_{\alpha_1,\beta_1}\cup\cdots\cup U^2_{\alpha_n,\beta_n}$. Then we have that $U(F_1)\cap U(F_2)=\varnothing$ and the definition of the topology $\tau_B^{S}$ implies that $U(F_1)$ and $U(F_2)$ are open neighbourhoods of closed sets $F_1$ and $F_2$ in $\left(B^0_{\lambda}(S),\tau_B^{S}\right)$, respectively. This completes the proof of the statement.
\end{proof}

Further we shall need the following lemma.

\begin{lemma}\label{lemma-2.13}
Let $S$ be a semigroup with the zero $0_S$. If a Hausdorff semitopological semigroup $T$ contains $S$ as a dense subsemigroup then $0_S$ is zero in $T$.
\end{lemma}

\begin{proof}
On the contrary, suppose that $0_S$ is not zero in $T$. Then there exists $t\in T$ such that either $t\cdot 0_S\neq 0_S$ or $0_S\cdot t\neq 0_S$. If $t\cdot 0_S=x\neq 0_S$ then the Hausdorffness of $T$ implies that there exist open neighbourhoods $U(0_S)$ and $U(x)$ of the points $0_S$ and $x$ in $T$, respectively, such that $U(0_S)\cap U(x)=\varnothing$. Also, the separate continuity of the semigroup operation in $T$ implies that there exists an open neighbourhood $V(t)$ of the point $t$ in $T$ such that $V(t)\cdot 0_S\subseteq U(x)$. Since $S$ is a dense subsemigroup of $T$ we conclude that $V(t)\cap S\neq\varnothing$ and hence we have that $0_S\in (V(t)\cap S)\cdot 0_S\subseteq V(t)\cdot 0_S\subseteq U(x)$, a contradiction. The obtained contradiction implies that $t\cdot 0_S=0_S$. The proof of the equality $0_S\cdot t=0_S$ is similar.
\end{proof}

In the same manner we can see that next lemma holds.

\begin{lemma}\label{lemma-2.14}
Let $S$ be a semigroup with the unit $1_S$. If a Hausdorff semitopological semigroup $T$ contains $S$ as a dense subsemigroup then $1_S$ is unit in $T$.
\end{lemma}

%%%%%%%%%%%%%%%%%%%%%%%%%%%%%%%%%%%%%%%%%%%%

\section{Stone-\v{C}ech and Bohr compactifications of pseudocompact Brandt $\lambda^0$-extensions of semitopological monoids}

We shall say that a Tychonoff space $X$ has the \emph{Grothendieck property} if for any Tychonoff space $Z$ every separately continuous map $f\colon X\times X\rightarrow Z$ can be extended to separately continuous map $\widehat{f}\colon\beta X\times\beta X\rightarrow\beta Z$. We observe that if $X, Y$ and $Z$ are Tychonoff pseudocompact spaces then every continuous map $f\colon X\times X\rightarrow Z$ can be extended to separately continuous map $\widehat{f}\colon\beta X\times\beta X\rightarrow\beta Z$ \cite{Reznichenko1994}. Also \cite{Arhangel'skij1984, Reznichenko1994}, a Tychonoff space $X$ has the Grothendieck property provides that $X$ is pseudocompact and one of the following conditions holds:
\begin{itemize}
  \item[$(i)$] $X$ is countably compact;
  \item[$(ii)$] $X$ has the countable tightness (i.e., $t(X)=\kappa$, where $t(X)$ is the minimal infinite cardinal $\kappa$ such that if $x\in X$, $A\subset X$ and $x\in\operatorname{cl}_X(A)$ then there exists $B\subseteq A$ such that $|B|\leqslant\kappa$ and $x\in\operatorname{cl}_X(B)$);
  \item[$(iii)$] $X$ is separable;
  \item[$(iv)$] $X$ is a $k$-space (i.e., $X$ is a Hausdorff space which can be represented as a quotient space of a locally compact space).
\end{itemize}

By Proposition~1.11 of \cite{Reznichenko1994} we have that if a Tychonoff space $X$ has the Grothendieck property then every separately continuous map $f\colon X\times X\rightarrow X$ extends to a separately continuous map $\widehat{f}\colon\beta X\times\beta X\rightarrow\beta X$, and hence the following proposition holds:

\begin{proposition}\label{proposition-2.15}
Let $S$ be a Tychonoff semitopological semigroup with the
Grothendieck property.  Then $\beta S$ is a compact
semitopological semigroup. Moreover, if $1_S$ is the unit of $S$
$($resp., $0_S$ is the zero of $S)$ then $1_S$ is the unit of
$\beta S$ $($resp., $0_S$ is the zero of $\beta S)$.
\end{proposition}

We observe that the last statement of Proposition~\ref{proposition-2.15} follows from  Lemmas~\ref{lemma-2.13} and ~\ref{lemma-2.14}.

\begin{theorem}\label{theorem-2.16}
Let $S$ be a Tychonoff semitopological monoid with the
Grothendieck property and zero. Then for an arbitrary cardinal
$\lambda\geqslant 1$ and the Tychonoff  pseudocompact topological
Brandt $\lambda^0$-extension
$\left(B^0_{\lambda}(S),\tau_{B}^S\right)$ of $S$ in the class of
semitopological semigroups the compact topological Brandt
$\lambda^0$-extension $\Big(B^0_{\lambda}(\beta S),\tau_{B}^{\beta
S}\Big)$ of $\beta S$ contains
$\left(B^0_{\lambda}(S),\tau_{B}^S\right)$.
\end{theorem}

\begin{proof}
By Proposition~\ref{proposition-2.15}, $\beta S$ is a compact semitopological monoid with zero $0_S$ and the unit $1_S$, which contains $S$ as a dense subsemigroup. We extend the embedding map $\beta\colon S\rightarrow\beta S$ onto $\beta_B^*\colon B^0_{\lambda}(S)\rightarrow B^0_{\lambda}(\beta S)$ by the formulae $\beta_B^*(0)=0$ and $\beta_B^*(\gamma,s,\delta)=(\gamma,\beta(s),\delta)$ for all $\gamma,\delta\in\lambda$. Then by Theorem~\ref{theorem-2.10}, $B^0_{\lambda}(\beta S)$ with the topology $\tau_{B}^{\beta S}$ is a compact topological Brandt $\lambda^0$-extension of $\beta S$ in the class of semitopological semigroups. It is easily seen that the map $\beta_B^*\colon B^0_{\lambda}(S)\rightarrow B^0_{\lambda}(\beta S)$ is an isomorphic topological embedding of the semitopological semigroup $\left(B^0_{\lambda}(S),\tau_{B}^S\right)$ into the compact semitopological semigroup $\left(B^0_{\lambda}(\beta S),\tau_{B}^{\beta S}\right)$.
\end{proof}

Theorem~\ref{theorem-2.16} implies the following corollary:

\begin{corollary}\label{corollary-2.17}
Let $S$ be a Tychonoff pseudocompact semitopological monoid with
zero such that one the following conditions holds:
\begin{itemize}
  \item[$(i)$] $S$ is countably compact;
  \item[$(ii)$] $S$ has the countable tightness;
  \item[$(iii)$] $S$ is separable;
  \item[$(iv)$] $S$ is a $k$-space.
\end{itemize}
Then for arbitrary cardinal $\lambda\geqslant 1$ and for the Tychonoff pseudocompact topological Brandt $\lambda^0$-extension $\left(B^0_{\lambda}(S),\tau_{B}^S\right)$ of $S$ in the class of semitopological semigroups the compact topological Brandt $\lambda^0$-extension $\Big(B^0_{\lambda}(\beta S),\tau_{B}^{\beta S}\Big)$ of $\beta S$ contains $\left(B^0_{\lambda}(S),\tau_{B}^S\right)$.
\end{corollary}

If $X$ is a topological space, then by $d(X)$ we denote the density of $X$, i.e.,
\begin{equation*}
d(X)=\min\left\{|A|\mid A \hbox{ is a dense subset of } X\right\}.
\end{equation*}

\begin{proposition}\label{proposition-2.20}
Let $\lambda$ by any cardinal $\geqslant 1$ and $S$ be a Hausdorff semitopological semigroup. Then the following statements hold:
\begin{itemize}
  \item[$(i)$] $t\left(B^0_{\lambda}(S),\tau_{B}^{S}\right)=t(S)$;
  \item[$(ii)$] $d\left(B^0_{\lambda}(S),\tau_{B}^{S}\right)=\max\left\{\lambda, d(S)\right\}$;
  \item[$(iii)$] $\left(B^0_{\lambda}(S),\tau_{B}^{S}\right)$ is separable if and only if $S$ is separable and $\lambda\leqslant\omega$;
  \item[$(iv)$] $\left(B^0_{\lambda}(S),\tau_{B}^{S}\right)$ is $k$-space if and only if $S$ is $k$-space.
\end{itemize}
\end{proposition}

\begin{proof}
$(i)$ Since $S$ is homeomorphic to the subspace $S_{i,j}$ of $\left(B^0_{\lambda}(S),\tau_{B}^{S}\right)$ for all $i,j\in\lambda$ we conclude that the definition of tightness of a topological space implies the following inequality $t(S)\leqslant t\left(B^0_{\lambda}(S),\tau_{B}^{S}\right)$.

Suppose that $t(S)=\kappa\geqslant\omega$. Then for every open subset $A$ of $S$ and every point $x\in\operatorname{cl}_{S}(A)$ there exists $B\subseteq A$ such that $|B|\leqslant\kappa$ and $x\in\operatorname{cl}_{S}(B)$.

Fix an arbitrary subset $A$ of $\left(B^0_{\lambda}(S),\tau_{B}^{S}\right)$. Let $x$ be an arbitrary point of $\left(B^0_{\lambda}(S),\tau_{B}^{S}\right)$ such that
$x\in\operatorname{cl}_{B^0_{\lambda}(S)}(A)$. We consider two cases:
\begin{itemize}
  \item[$(1)$] $x\in S_{i,j}^*$ for some $i,j\in\lambda$;
  \item[$(2)$] $x$ is zero of the semigroup $B^0_{\lambda}(S)$.
\end{itemize}

Suppose that case $(1)$ holds. By Lemma~\ref{lemma-2.2}, $S_{i,j}^*$ is an open subspace of $\left(B^0_{\lambda}(S),\tau_{B}^{S}\right)$ and since $t(S_{i,j}^*)\leqslant t(S)$ we conclude that there exists a subset $B\subseteq A\cap S_{i,j}^*\subseteq A$ such that $x\in\operatorname{cl}_{B^0_{\lambda}(S)}(B)$ and $|B|\leqslant\kappa$.

Suppose that $x$ is zero of the semigroup $B^0_{\lambda}(S)$ and $x\in\operatorname{cl}_{B^0_{\lambda}(S)}(A)$ for some $A\subseteq B^0_{\lambda}(S)$. If the set $A$ intersects infinitely many sets $S_{i,j}^*$, $i,j\in\lambda$, then  there exists an infinite countable subset $B$ of $B^0_{\lambda}(S)$ such that the following conditions holds:
\begin{itemize}
  \item[$(1)$] either $B\cap S_{i,j}^*=\varnothing$ or $B\cap S_{i,j}^*$ is a singleton for all $i,j\in\lambda$; \; and
  \item[$(2)$] $B\subseteq A$.
\end{itemize}
Then the definition of the topology $\tau_{B}^{S}$ on $B^0_{\lambda}(S)$ implies that $x\in\operatorname{cl}_{B^0_{\lambda}(S)}(B)$. In the other case there exist finitely many subsets $S_{i_1,j_1}^*,\ldots,S_{i_k,j_k}^*$ of $B^0_{\lambda}(S)$ such that $A\subseteq S_{i_1,j_1}^*\cup\cdots\cup S_{i_k,j_k}^*\cup\{x\}$ and $x\in\operatorname{cl}_{S_{i_1,j_1}}\left(A\cap S_{i_1,j_1}^*\right),\ldots, x\in \operatorname{cl}_{S_{i_k,j_k}}\left(A\cap S_{i_k,j_k}^*\right)$. Since $t(S)\leqslant \kappa$ we conclude that there exist $B_{i_1,j_1}\subseteq S_{i_1,j_1}^*, \ldots, B_{i_k,j_k} \subseteq S_{i_k,j_k}^*$ such that $\left|B_{i_1,j_1}\right|\leqslant\kappa,\ldots, \left|B_{i_k,j_k}\right|\leqslant\kappa$ and $x\in\operatorname{cl}_{S_{i_1,j_1}}\left(B_{i_1,j_1}\right)$, $\ldots$, $x\in \operatorname{cl}_{S_{i_k,j_k}}\left(B_{i_k,j_k}\right)$. Then the definition of the topology $\tau_{B}^{S}$ on $B^0_{\lambda}(S)$ implies that $t\left(B^0_{\lambda}(S),\tau_{B}^{S}\right)\leqslant\kappa$.

$(ii)$ Let $A$ be a dense subset of the topological space $S$. Then the definition of the topology $\tau_{B}^{S}$ on $B^0_{\lambda}(S)$ implies that $A_B=\displaystyle\bigcup_{i,j\in\lambda}A_{i,j}$ is a dense subset of the topological space $\left(B^0_{\lambda}(S),\tau_{B}^{S}\right)$, and hence we have that $d\left(B^0_{\lambda}(S),\tau_{B}^{S}\right)\leqslant\max\left\{\lambda, d(S)\right\}$.

By Lemma~\ref{lemma-2.1} we have that the space $S$ is homeomorphic to the subspace $S_{i,j}$ of $\left(B^0_{\lambda}(S),\tau_{B}^{S}\right)$ for all $i,j\in\lambda$, and hence we get that $d(S)=d(S_{i,j})$. If $A$ is a dense subset of the topological space $\left(B^0_{\lambda}(S),\tau_{B}^{S}\right)$ then since by Lemma~\ref{lemma-2.2}, $S_{i,j}^*$ is an open subset of $\left(B^0_{\lambda}(S),\tau_{B}^{S}\right)$ we have that $A\cap S_{i,j}^*$ is a dense subset of $S_{i,j}^*$, and hence we get that $d(S_{i,j}^*)\leqslant d\left(B^0_{\lambda}(S),\tau_{B}^{S}\right)$ for all $i,j\in\lambda$. Then the set $\{0\}\cup(A\cap S_{i,j}^*)$ is dense in $S_{i,j}$, and hence we get that $d(S)=d(S_{i,j})\leqslant d\left(B^0_{\lambda}(S),\tau_{B}^{S}\right)$. Also, since $S_{i,j}^*$ is an open subset of the topological space $\left(B^0_{\lambda}(S),\tau_{B}^{S}\right)$ we conclude that if $A$ is a dense subset of $\left(B^0_{\lambda}(S),\tau_{B}^{S}\right)$ then $\lambda\leqslant|A|$. This implies the inequality $\lambda\leqslant d\left(B^0_{\lambda}(S),\tau_{B}^{S}\right)$ which completes the proof of the statement.

Statement $(iii)$ follows from $(ii)$.

$(iv)$ If $\left(B^0_{\lambda}(S),\tau_{B}^{S}\right)$ is a $k$-space then the definition of the topology $\tau_{B}^{S}$ implies that $S_{i,j}$ is a closed subset of $\left(B^0_{\lambda}(S),\tau_{B}^{S}\right)$ for all $i,j\in\lambda$, and hence by Theorem~3.3.25 from \cite{Engelking1989} and Lemma~\ref{lemma-2.1} we have that $S$ is a $k$-space.

Suppose that $S$ is a $k$-space and $\mathscr{A}(\lambda)$ is the one-point Alexandroff compactification of the infinite discrete space $\lambda$ with the remainder $a$. Then by Theorem~3.3.27 from \cite{Engelking1989} we have that $\mathscr{A}(\lambda)\times S$ with the product topology is a $k$-space. We put
\begin{equation*}
    \mathscr{J}=\{(a,s)\mid s\in S\}\cup\left\{(i,0_S)\mid i\in\lambda \hbox{ and } 0_S \hbox{ is zero of } S\right\}.
\end{equation*}
Then the definition of the space $\mathscr{A}(\lambda)\times S$ implies that $\mathscr{J}$ is a closed subset of  $\mathscr{A}(\lambda)\times S$.

Simple verifications show that the relation $\Omega_{\mathscr{J}}= (\mathscr{J}{\times}\mathscr{J})\cup\Delta_{\mathscr{A}(\lambda)\times S}$, where $\Delta_{\mathscr{A}(\lambda)\times S}$ is the diagonal of $\mathscr{A}(\lambda)\times S$, is an equivalence relation on $\mathscr{A}(\lambda)\times S$.

Next we shall show that $\Omega_{\mathscr{J}}$ is a closed relation on $\mathscr{A}(\lambda)\times S$. Fix an arbitrary $((i,s),(j,t))\in \left(\left(\mathscr{A}(\lambda)\times S\right)\times \left(\mathscr{A}(\lambda)\times S\right)\right)\setminus\Omega_{\mathscr{J}}$. Since $\Omega_{\mathscr{J}}$ is symmetric we need only to consider the following cases:
\begin{itemize}
  \item[$(1)$] $i=a$, $j\neq a$, and $t\neq 0_S$;
  \item[$(2)$] $i=j\neq a$, $s=0_S$, and $t\neq 0_S$;
  \item[$(3)$] $i=j\neq a$ and $0_S\neq s\neq t\neq 0_S$;
  \item[$(4)$] $a\neq i\neq j\neq a$, $s\neq 0_S$, and $t\neq 0_S$.
\end{itemize}

In case $(1)$ the Hausdorffness of $S$ implies that there exists an open neighbourhood $U(t)$ of the point $t$ in $S$ such that $U(t)\not\ni 0_S$. Then the definition of the topological space $\mathscr{A}(\lambda)$ implies that $\left(\mathscr{A}(\lambda)\setminus\{j\}\right)\times S$ and $\{j\}\times U(t)$ are open disjoint neighbourhoods of the points $(i,s)$ and $(j,t)$ in $\mathscr{A}(\lambda)\times S$, respectively, and hence $((i,s),(j,t))$ is an interior point of the set $\left(\left(\mathscr{A}(\lambda)\times S\right)\times \left(\mathscr{A}(\lambda)\times S\right)\right)\setminus\Omega_{\mathscr{J}}$.

In case $(2)$ the Hausdorffness of $S$ implies that there exist an open neighbourhoods $U(s)$ and $U(t)$ of the points $s=0_S$ and $t$ in $S$, respectively, such that $U(s)\cap U(t)=\varnothing$. Since $i=j\neq a$ is an isolated point of the space $\mathscr{A}(\lambda)$ we have that $\{i\}\times U(s)$ and $\{j\}\times U(t)$ are open disjoint neighbourhoods of the points $(i,s)$ and $(j,t)$ in $\mathscr{A}(\lambda)\times S$, respectively, and hence $((i,s),(j,t))$ is an interior point of the set $\left(\left(\mathscr{A}(\lambda)\times S\right)\times \left(\mathscr{A}(\lambda)\times S\right)\right)\setminus\Omega_{\mathscr{J}}$.

In cases $(3)$ and $(4)$ the proof of the statement that $((i,s),(j,t))$ is an interior point of the set $\left(\left(\mathscr{A}(\lambda)\times S\right)\times \left(\mathscr{A}(\lambda)\times S\right)\right)\setminus\Omega_{\mathscr{J}}$ are similar to $(2)$ and $(1)$, respectively.

Therefore we have that $\Omega_{\mathscr{J}}$ is a closed equivalence relation on the topological space $\mathscr{A}(\lambda)\times S$ and hence the quotient space $\left(\mathscr{A}(\lambda)\times S\right)/\Omega_{\mathscr{J}}$ (with the quotient topology) is a Hausdorff space. Thus, the natural map $\pi_{\Omega_{\mathscr{J}}}\colon \mathscr{A}(\lambda)\times S\rightarrow \left(\mathscr{A}(\lambda)\times S\right)/\Omega_{\mathscr{J}}$ is a quotient map and by Theorem~3.3.23 from \cite{Engelking1989} we get that $\left(\mathscr{A}(\lambda)\times S\right)/\Omega_{\mathscr{J}}$ is a $k$-space for every infinite cardinal $\lambda$. Also, we observe that for every finite subset $K_{\texttt{fin}}\subset \mathscr{A}(\lambda)$ such that $a\in K_{\texttt{fin}}$ we have that $\pi_{\Omega_{\mathscr{J}}}(K_{\texttt{fin}}\times S)$ is a $k$-space by Theorem~3.3.23 of \cite{Engelking1989}, because the definition of the space $\mathscr{A}(\lambda)\times S$ implies that $A\times S^*$ is an open subspace of $\mathscr{A}(\lambda)\times S$ for every subset $A\subseteq \mathscr{A}(\lambda) \setminus \{a\}$.

We observe that if $\lambda$ is an infinite cardinal then $|\lambda\times\lambda|= \lambda$ and simple verifications show that the map $f\colon \left(B^0_{\lambda}(S),\tau_{B}^{S}\right) \rightarrow \left(\mathscr{A}(\lambda)\times S\right)/\Omega_{\mathscr{J}}$ defined by the formulae
\begin{equation*}
    f(i,s,j)=(i,j,s)\in\lambda\times\lambda\times S^*, \hbox{ if } s\in S^* \quad \hbox{and} \quad f(0)=\pi_{\Omega_{\mathscr{J}}}(\mathscr{J}\times\mathscr{J})
\end{equation*}
is a homeomorphism in the case when we identify cardinal $\lambda$ with its square $\lambda\times\lambda$. This completes the proof of statement $(iv)$.
\end{proof}

Recall~\cite{DeLeeuwGlicksberg1961} that a \emph{Bohr compactification} of a semitopological semigroup $S$ is a~pair $(b, \mathbb{B}(S))$ such that $\mathbb{B}(S)$ is a compact semitopological semigroup, $b\colon S\to \mathbb{B}(S)$ is a continuous homomorphism, and if $g\colon S\to T$ is a continuous homomorphism of $S$ into a compact semitopological semigroup $T$, then there exists a unique continuous homomorphism $f\colon \mathbb{B}(S)\to T$ such that the diagram
\begin{equation*}
\xymatrix{ S\ar[rr]^b\ar[dr]_g && \mathbb{B}(S)\ar[ld]^f\\
& T &}
\end{equation*}
commutes. In the sequel, similar as in the general topology by the Bohr compactification of a semitopological semigroup $S$ we shall mean not only pair $(b,
\mathbb{B}(S))$ but also the compact semitopological semigroup $\mathbb{B}(S)$.

The definition of the Stone-\v{C}ech compactification implies that for every Tychonoff semitopological monoid $S$ with the Grothendieck property any continuous homomorphism $h\colon S\rightarrow T$ into a compact semitopological semigroup $T$ has the unique extended continuous homomorphism $\beta h\colon\beta S\rightarrow T$. This implies that the Bohr compactification of $S$ in this case coincides with its Stone-\v{C}ech compactification $\beta S$. Hence Proposition~\ref{proposition-2.20} implies the following two theorems.

\begin{theorem}\label{theorem-2.21}
Let $S$ be a Tychonoff pseudocompact semitopological monoid with
zero such that one the following conditions holds:
\begin{itemize}
  \item[$(i)$] $S$ is countably compact;
  \item[$(ii)$] $S$ has the countable tightness;
  \item[$(iii)$] $S$ is a $k$-space.
\end{itemize}
Then for arbitrary cardinal $\lambda\geqslant 1$ and for the Tychonoff pseudocompact topological Brandt $\lambda^0$-extension $\left(B^0_{\lambda}(S),\tau_{B}^S\right)$ of $S$ in the class of semitopological semigroups the Stone-\v{C}ech compactification $\beta\left(B^0_{\lambda}(S),\tau_{B}^S\right)$ is a compact semitopological semigroup and, moreover, the Bohr compactification $\mathbb{B}\left(B^0_{\lambda}(S),\tau_{B}^S\right)$ of $\left(B^0_{\lambda}(S),\tau_{B}^S\right)$ is topologically isomorphic to $\beta\left(B^0_{\lambda}(S),\tau_{B}^S\right)$.
\end{theorem}

\begin{theorem}\label{theorem-2.22}
Let $\lambda\geqslant 1$ be a countable cardinal and $S$ be a Tychonoff separable pseudocompact semitopological monoid with zero. Then the Stone-\v{C}ech compactification $\beta\left(B^0_{\lambda}(S),\tau_{B}^S\right)$ of the Tychonoff pseudocompact topological Brandt $\lambda^0$-extension $\left(B^0_{\lambda}(S),\tau_{B}^S\right)$ of $S$ in the class of semitopological semigroups is a compact semitopological semigroup and, moreover, the Bohr compactification $\mathbb{B}\left(B^0_{\lambda}(S),\tau_{B}^S\right)$ of $\left(B^0_{\lambda}(S),\tau_{B}^S\right)$ is topologically isomorphic to $\beta\left(B^0_{\lambda}(S),\tau_{B}^S\right)$.
\end{theorem}

The following theorem is an analogue of the Comfort-Ross Theorem for topological group (see \cite{ComfortRoss1966}).

\begin{theorem}\label{theorem-2.23}
Let $\lambda\geqslant 1$ be any cardinal and $S$ be a normal countably compact semitopological monoid with zero. Then the Stone-\v{C}ech compactification $\beta\left(B^0_{\lambda}(S),\tau_{B}^S\right)$ of a countably compact topological Brandt $\lambda^0$-extension $\left(B^0_{\lambda}(S),\tau_{B}^S\right)$ is topologically isomorphic to the compact topological Brandt $\lambda^0$-extension $\left(B^0_{\lambda}(\beta S),\tau_{B}^{\beta S}\right)$ of $\beta S$.
\end{theorem}

\begin{proof}
Theorem~\ref{theorem-2.21} implies that $\beta\left(B^0_{\lambda}(S),\tau_{B}^S\right)$ is a compact semitopological semigroup which contains $\left(B^0_{\lambda}(S),\tau_{B}^S\right)$ as a dense subsemigroup. By Theorem~\ref{theorem-2.7c} we have that $\tau_{B}^S$ is a unique Hausdorff topology on $B^0_{\lambda}(S)$ such that $\left(B^0_{\lambda}(S),\tau_{B}^S\right)$ is a countably compact semitopological semigroup and by Proposition~\ref{proposition-2.19}~$(iii)$ the space $\left(B^0_{\lambda}(S),\tau_{B}^S\right)$ is normal. Now, Corollary~3.6.8 of \cite{Engelking1989} implies that $\operatorname{cl}_{\beta\left(B^0_{\lambda}(S),\tau_{B}^S\right)}(S_{\iota,\kappa})= \beta (S_{\iota,\kappa})=\beta S$ for all $\iota,\kappa\in\lambda$. We observe that the separate continuity of the semigroup operation in $\beta\left(B^0_{\lambda}(S),\tau_{B}^S\right)$ implies that $(\iota,1_S,\iota)\cdot x\cdot (\kappa,1_S,\kappa)=x$ for all $x\in S_{\iota,\kappa}\subset \beta(S_{\iota,\kappa})\subset \beta\left(B^0_{\lambda}(S),\tau_{B}^S\right)$, where $\iota,\kappa\in\lambda$.

Fix an arbitrary $x\in\beta\left(B^0_{\lambda}(S),\tau_{B}^S\right)\setminus B^0_{\lambda}(S)$. Then the Hausdorffness of $\beta\left(B^0_{\lambda}(S),\tau_{B}^S\right)$ and Theorem~\ref{theorem-2.7c} imply that there exists an open neighbourhood $U(x)$ of the point $x$ in $\beta\left(B^0_{\lambda}(S),\tau_{B}^S\right)$ which intersects finitely many sets $S^*_{\iota,\kappa}$, $\iota,\kappa\in\lambda$. Next we shall show that there exist $\iota_x,\kappa_x\in\lambda$ and open neighbourhood $V(x)$ of the point $x$ in $\beta\left(B^0_{\lambda}(S),\tau_{B}^S\right)$ such that $V(x)\cap B^0_{\lambda}(S)\subseteq S^*_{\iota_x,\kappa_x}$. Suppose to the contrary that $V(x)\cap B^0_{\lambda}(S)\nsubseteq S^*_{\iota,\kappa}$ for any open neighbourhood $V(x)$ of the point $x$ in $\beta\left(B^0_{\lambda}(S),\tau_{B}^S\right)$ and any pair of indices  $(\iota,\kappa)$. Then there exists at least one pair $(\gamma,\delta)\neq(\iota,\kappa)$ such that $V(x)\cap S^*_{\gamma,\delta}\neq\varnothing$. Then Lemma~\ref{lemma-2.13} implies that the zero $0$ of $B^0_{\lambda}(S)$ is zero of $\beta\left(B^0_{\lambda}(S),\tau_{B}^S\right)$ and hence $x\neq 0$. If $(\iota,1_S,\iota)\cdot x\cdot(\kappa,1_S,\kappa)=0$, then there exist open neighbourhoods $W(0)$ and $W(x)$ of the points $0$ and $x$ in $\beta\left(B^0_{\lambda}(S),\tau_{B}^S\right)$, respectively, such that the following conditions hold:
\begin{equation*}
    W(0)\cap W(x)=\varnothing, \quad W(x)\subseteq V(x) \quad \hbox{and}\quad
    (\iota,1_S,\iota)\cdot W(x)\cdot(\kappa,1_S,\kappa)\subseteq W(0).
\end{equation*}
But $W(x)\cap S_{\iota,\kappa}\neq\varnothing$ and hence we have that $\left((\iota,1_S,\iota)\cdot W(x)\cdot(\kappa,1_S,\kappa)\right)\cap W(x)\neq\varnothing$, a contradiction. If $(\iota,1_S,\iota)\cdot x\cdot(\kappa,1_S,\kappa)=y\neq 0$, then for every open neighbourhood $W(y)\not\ni 0$ of the point $y$ in $\beta\left(B^0_{\lambda}(S),\tau_{B}^S\right)$ there exists an open neighbourhood $W(x)$ of $x$ in $\beta\left(B^0_{\lambda}(S),\tau_{B}^S\right)$ such that $(\iota,1_S,\iota)\cdot W(x)\cdot(\kappa,1_S,\kappa)\subseteq W(y)$. Since $W(x)\cap S^*_{\gamma,\delta}$ for some $(\gamma,\delta)\neq(\iota,\kappa)$ we get that $0\in(\iota,1_S,\iota)\cdot W(x)\cdot(\kappa,1_S,\kappa)\subseteq W(y)$, a contradiction. The obtained contradictions imply that $x\in \operatorname{cl}_{\beta\left(B^0_{\lambda}(S),\tau_{B}^S\right)}(S_{\iota,\kappa})= \beta (S_{\iota,\kappa})$ for some $\iota,\kappa\in\lambda$.

The separate continuity of the semigroup operation in $\beta\left(B^0_{\lambda}(S),\tau_{B}^S\right)$ and Lemma~\ref{lemma-2.13} imply that \begin{equation*}
    \beta(S_{\iota,\kappa})\cap\beta(S_{\gamma,\delta})=
    \left\{
      \begin{array}{ll}
        \beta(S_{\iota,\kappa}), & \hbox{if~} (\iota,\kappa)=(\gamma,\delta)\\
        \{0\}, & \hbox{if~} (\iota,\kappa)\neq(\gamma,\delta),
      \end{array}
    \right.
\end{equation*}
and $(\iota,1_S,\iota)\cdot x\cdot(\kappa,1_S,\kappa)=x$ for any $x\in\beta(S_{\iota,\kappa})$ and all $\iota,\kappa,\gamma,\delta\in\lambda$. Next, we suppose that the map $f_{\iota,\kappa}\colon\beta S\rightarrow \beta(S_{\iota,\kappa})$ identifies the compact spaces $\beta S$ and $\beta(S_{\iota,\kappa})$ such that $f_{\iota,\kappa}(1_S)=(\iota,1_S,\kappa)\in \beta(S_{\iota,\kappa})$ and $f_{\iota,\kappa}(0_S)=0\in \beta(S_{\iota,\kappa})\subset \beta\left(B^0_{\lambda}(S),\tau_{B}^S\right)$. Then the map $f\colon B^0_{\lambda}(\beta S)\rightarrow \beta\left(B^0_{\lambda}(S),\tau_{B}^S\right)$ defined by the formula
\begin{equation*}
    f(x)=
\left\{
  \begin{array}{ll}
    f_{\iota,\kappa}(x), & \hbox{if~} x\in(\beta S)_{\iota,\kappa}; \\
    0, & \hbox{if~} x=0,
  \end{array}
\right.
\end{equation*}
is an algebraic isomorphism. Therefore, $\beta\left(B^0_{\lambda}(S),\tau_{B}^S\right)$ is a compact Brandt $\lambda^0$-extension of $\beta S$. Theorem~\ref{theorem-2.7c} completes the proof.
\end{proof}

\begin{remark}\label{remark-2.24}
We observe that by Theorem~3.10.21 from \cite{Engelking1989} every normal pseudocompact topological space is countably compact and hence Theorem~\ref{theorem-2.23} holds in the case of pseudocompact topological Brandt $\lambda^0$-extensions of normal semitopological semigroups.
\end{remark}
%%%%%%%%%%%%%%%%%%%%%%%%%%%%%%%%%%%%%%%%%%%%%%%%%%%%%%%%%%%%%%%%%%%%%

\section{On categories of pseudocompact topological Brandt
$\lambda^0$-extensions of semitopological semigroups}

\begin{definition}[\cite{GutikRepovs2010}]\label{definition-3.1}
Let $\lambda$ be any cardinal $\geqslant 2$. We shall say that a semigroup $S$ has the \emph{$\mathcal{B}_{\lambda}^*$-property} if $S$ satisfies the following conditions:
\begin{itemize}
    \item[1)] $S$ does not contain the semigroup of $\lambda\times\lambda$-matrix units;
    \item[2)] $S$ does not contain the semigroup of $2\times 2$-matrix units $B_2$ such that the zero of $B_2$ is the zero of $S$.
\end{itemize}
\end{definition}

\begin{theorem}\label{theorem-3.2}
Let $\lambda_1$ and $\lambda_2$ be any cardinals such that
$\lambda_2\geqslant\lambda_1\geqslant 1$. Let $B_{\lambda_1}^0(S)$
and $B_{\lambda_2}^0(T)$ be topological Brandt $\lambda_1^0$- and
$\lambda_2^0$-extensions of semiregular pseudocompact
semitopological monoids $S$ and $T$ with zero, respectively. Let
$h\colon S\rightarrow T$ be a continuous homomorphism such that
$h(0_S)=0_T$ and $\varphi\colon {\lambda_1}\rightarrow{\lambda_2}$
an one-to-one map. Let $e$ be a non-zero idempotent of $T$, $H_e$
a maximal subgroup of $T$ with unit $e$ and
$u\colon{\lambda_1}\rightarrow H_e$ a map. Then $I_h=\{ s\in S\mid
h(s)=0_T\}$ is a closed ideal of $S$ and the map $\sigma\colon
B_{\lambda_1}^0(S)\rightarrow B_{\lambda_2}^0(T)$ defined by the
formulae
\begin{equation*}
    \sigma((\alpha,s,\beta))=
    \left\{%
\begin{array}{cl}
    (\varphi(\alpha),u(\alpha)\cdot h(s)\cdot(u(\beta))^{-1},\varphi(\beta)),
    & \hbox{if}\quad s\in S\setminus I_h ;\\
    0_2, & \hbox{if}\quad s\in I_h^*,\\
\end{array}%
\right.
\end{equation*}
and $(0_1)\sigma=0_2$, is a non-trivial continuous homomorphism from $B_{\lambda_1}^0(S)$ into $B_{\lambda_2}^0(T)$, where $0_1$ and $0_2$ are zeros of $B_{\lambda_1}^0(S)$ and $B_{\lambda_2}^0(T)$, respectively. Moreover, if for the semigroup $T$ the following conditions hold:
\begin{itemize}
    \item[($i$)] every idempotent of $T$ lies in the center of $T$;
    \item[($ii$)] $T$ has $\mathcal{B}_{\lambda_1}^*$-property,
\end{itemize}
then every non-trivial continuous homomorphism from
$B_{\lambda_1}^0(S)$ into $B_{\lambda_2}^0(T)$ can be constructed
in this manner.
\end{theorem}

\begin{proof}
The algebraic part of the proof follows from Theorem~3.10 of \cite{GutikRepovs2010}.

Since the homomorphism $h$ is continuous, $I_h=h^{-1}(0_T)$ is a closed ideal of the semitopological semigroup $S$.

Our next goal is to show that the homomorphism $\sigma$ is continuous provides $h$ is continuous. We consider the following cases:
\begin{itemize}
    \item[$(i)$] $\sigma(0_1)=0_2$;
    \item[$(ii)$] $\sigma((\alpha, s,\beta))=0_2$, i.e., $s\in I_h^*$; and
    \item[$(iii)$] $\sigma((\alpha, s,\beta))=
    (\varphi(\alpha),u(\alpha)\cdot h(s)\cdot(u(\beta))^{-1},\varphi(\beta))$,
\end{itemize}
where $(\alpha, s,\beta)$ is any non-zero element of the semigroup $B_{\lambda_1}^0(S)$.

Without loss of generality we may assume that $\varphi\colon\lambda_1\rightarrow \lambda_2$ is a bijection. Moreover, to simplify of the proof we can assume that
$(\alpha)\varphi=\alpha$ for all $\alpha\in\lambda_1$.

We begin from case $(i)$. Let $U_{A}(0_2)=\bigcup_{(\alpha,\beta)\in(\lambda_2\times\lambda_2) \setminus A}S_{\alpha,\beta}\cup \bigcup_{(\gamma,\delta)\in A} (U(0_T))_{\gamma,\delta}$ be any open basic neighbourhood of the zero $0_2$ in $B_{\lambda_2}^0(T)$. Since left and right translations in $T$ and the homomorphism $h\colon S\rightarrow T$ are continuous maps, we have that for any $u(\alpha),(u(\beta))^{-1}\in B_{\lambda_1}^0(S)$ there exists an open neighbourhood $V^{\alpha,\beta}(0_S)$ in $S$ such that $u(\alpha)\cdot h(V^{\alpha,\beta}(0_S))\cdot (u(\beta))^{-1}\subseteq U(0_T)$. Put $V(0_S)=\bigcap_{(\alpha,\beta)\in A} V^{\alpha,\beta}(0_S)$ and
$V_A(0_1)=\bigcup_{(\alpha,\beta)\in(\lambda_1\times\lambda_1) \setminus A}S_{\alpha,\beta}\cup \bigcup_{(\gamma,\delta)\in A} (V(0_S))_{\gamma,\delta}$. Then we get that $\sigma(V_A(0_1))\subseteq U_A(0_2)$.

In case $(ii)$ we have that $h(s)=0_T$. Let $U_{A}(0_2)=\bigcup_{(\alpha,\beta)\in(\lambda_2\times\lambda_2) \setminus A}S_{\alpha,\beta}\cup \bigcup_{(\gamma,\delta)\in A} (U(0_T))_{\gamma,\delta}$ be any basic open neighbourhood of the zero $0_2$ in $B_{\lambda_2}^0(T)$. Since left and right
translations in $T$ and the homomorphism $h\colon S\rightarrow T$
are continuous maps, we conclude that for the open neighbourhood $U(0_T)$ of the
zero $0_T$ in $T$ there exists an open neighbourhood $V(s)$ of the point $s$ in $S$
such that $u(\alpha)\cdot h(V(s))\cdot (u(\beta))^{-1}\subseteq U(0_T)$. Therefore we have that $\sigma((V(s))_{\alpha,\beta})\subseteq U_A(0_2)$.

Next we consider case $(iii)$. Let $U_{\alpha,\beta}$ be any basic open neighbourhood of the point $\sigma((\alpha, s,\beta))= (\alpha,(\alpha)u\cdot(s)h\cdot((\beta)u)^{-1},\beta)$ in the topological space $B_{\lambda_2}^0(T)$. Since left and right translations in the
semigroup $T$ and the homomorphism $h\colon S\rightarrow T$ are continuous maps, we conclude that there exists an open neighbourhood $V(s)$ of the point $s$ in $S$ such that $u(\alpha)\cdot h(V(s))\cdot (u(\beta))^{-1}\subseteq U$ and hence we get that
$\sigma((V(s))_{\alpha,\beta})\subseteq U_{\alpha,\beta}$.

Since left and right translations in the topological semigroup $B_{\lambda_2}^0(T)$ are continuous and any restriction of a continuous map is continuous, we conclude that the continuity of the homomorphism $\sigma\colon B_{\lambda_1}^0(S)\rightarrow
B_{\lambda_2}^0(T)$ implies the continuity of $h$.
\end{proof}

Next we define a category of pairs of semiregular pseudocompact semitopological monoids with zero and sets, and a category of semiregular pseudocompact semitopological semigroups.

Let $S$ and $T$ be semiregular pseudocompact semitopological monoids with zeros. Let $\texttt{CHom}\,_0(S,T)$ be a set of all continuous homomorphisms $\sigma\colon S\rightarrow T$ such that $\sigma(0_S)=0_T$. We put
\begin{equation*}
    \mathbf{E}^{\textit{top}}_1(S,T)=\{e\in E(T)\mid \hbox{there exists~} \sigma\in
    \texttt{CHom}\,_0(S,T) \hbox{~such that~} \sigma(1_S)=e\}
\end{equation*}
and define the family
\begin{equation*}
    \mathscr{H}^{\textit{top}}_1(S,T)=\{ H(e)\mid
    e\in\mathbf{E}^{\textit{top}}_1(S,T)\},
\end{equation*}
where by $H(e)$ we denote the maximal subgroup with the unit $e$ in the semigroup $T$. Also by $\mathfrak{PCTB}$ we denote the class of all semiregular pseudocompact semitopological monoids $S$ with zero such that $S$ has $\mathcal{B}^*$-property and every idempotent of
$S$ lies in the center of $S$.

We define a category $\mathscr{T\!PC\!B}$ as
follows:
\begin{itemize}
    \item[$(i)$] $\operatorname{\textbf{Ob}}(\mathscr{T\!PC\!B})=\{(S,\lambda)\mid S\in\mathfrak{PCTB} \textrm{~and~} \lambda\textrm{~is a nonzero cardinal}\}$, and if $S$ is a trivial semigroup then we identify $(S,\lambda_1)$ and $(S,\lambda_2)$ for all nonzero cardinals $\lambda_1$ and $\lambda_2$;
    \item[$(ii)$] $\operatorname{\textbf{Mor}}(\mathscr{T\!PC\!B})$ consists of triples $(h,u,\varphi)\colon (S,\lambda)\rightarrow(S^{\,\prime},\lambda^{\,\prime})$, where
\begin{equation*}
\begin{split}
    & h\colon S\rightarrow S^{\,\prime} \textrm{~is a continuous homomorphism such that~} h\in \texttt{CHom}\,_0(S,S^{\,\prime}),\\
    & u\colon \lambda\rightarrow H(e) \textrm{~is a map}, \textrm{~for~} H(e)\in
      \mathscr{H}^{\textit{top}}_1(S,S^{\,\prime}),\\
    & \varphi\colon \lambda\rightarrow \lambda^{\,\prime} \textrm{~is an one-to-one map},
\end{split}
\end{equation*}
with the composition
\begin{equation*}\label{eq3-2}
    (h,u,\varphi)(h^{\,\prime},u^{\,\prime},\varphi^{\,\prime})=
    (hh^{\,\prime},[u,\varphi,h^{\,\prime},u^{\,\prime}],\varphi\varphi^{\,\prime}),
\end{equation*}
where the map $[u,\varphi,h^{\,\prime},u^{\,\prime}]\colon
\lambda\rightarrow H(e)$ is defined by the formula
\begin{equation*}
[u,\varphi,h^{\,\prime},u^{\,\prime}](\alpha)=
    u^{\,\prime}(\varphi(\alpha))\cdot h^{\,\prime}(u(\alpha))
    \qquad \textrm{~for~} \alpha\in\lambda.
\end{equation*}
\end{itemize}
Straightforward verification shows that $\mathscr{T\!PC\!B}$ is the category with the
identity morphism $\varepsilon_{(S,\lambda)}=(\operatorname{Id}_S,u_0,\operatorname{Id}_\lambda)$
for any $(S,\lambda)\in\operatorname{\textbf{Ob}}(\mathscr{T\!PC\!B})$, where $\operatorname{Id}_S\colon S\rightarrow S$ and $\operatorname{Id}_\lambda\colon \lambda\rightarrow \lambda$ are identity maps and $u_0(\alpha)=1_S$ for all $\alpha\in\lambda$.

We define a category $\mathscr{B}^*(\mathscr{T\!PC\!S})$ as follows:
\begin{itemize}
    \item[$(i)$]
    let
    $\operatorname{\textbf{Ob}}(\mathscr{B}^*(\mathscr{T\!PC\!S}))$ be all semiregular pseudocompact topological Brandt $\lambda^0$-extensions of semiregular pseudocompact semitopological monoids $S$ with zero in the class of semitopological semigroups such that $S$ has $\mathcal{B}^*$-property and every idempotent of $S$ lies in the center of $S$;
    \item[$(ii)$] let $\operatorname{\textbf{Mor}}(\mathscr{B}^*(\mathscr{T\!PC\!S}))$ be all homomorphisms of semiregular pseudocompact topological Brandt $\lambda^0$-ex\-ten\-sions of semiregular pseudocompact semitopological monoids $S$ with zero such that $S$ has $\mathcal{B}^*$-property and every idempotent of $S$ lies in the center of $S$.
\end{itemize}

For each $(S,\lambda_1)\in\operatorname{\textbf{Ob}}(\mathscr{T\!PC\!B})$
with non-trivial $S$, let $\textbf{B}(S,{\lambda_1})=B_{\lambda_1}^0(S)$ be the semiregular pseudocompact topological Brandt $\lambda^0_1$-extension of the semiregular pseudocompact semitopological monoid $S$ in the class of semitopological semigroups. For each
$(h,u,\varphi)\in\operatorname{\textbf{Mor}}(\mathscr{T\!PC\!B})$ with a non-trivial continuous homomorphism $h$, where $(h,u,\varphi)\colon
(S,{\lambda_1})\rightarrow(T,{\lambda_2})$ and
$(T,{\lambda_2})\in\operatorname{\textbf{Ob}}(\mathscr{T\!PC\!B})$, we define a map $\textbf{B}{(h,u,\varphi)}\colon \textbf{B}(S,{\lambda_1})=B_{\lambda_1}^0(S)\rightarrow
\textbf{B}(T,{\lambda_2})=B_{\lambda_2}^0(T)$ in the following way:
\begin{equation*}\label{functor_B}
    [\textbf{B}{(h,u,\varphi)}]((\alpha,s,\beta))=
    \left\{%
\begin{array}{cl}
    (\varphi(\alpha),u(\alpha)\cdot h(s)\cdot(u(\beta))^{-1},\varphi(\beta)),
    & \hbox{if}\quad s\in S\setminus I_h ;\\
    0_2, & \hbox{if}\quad s\in I_h^*,\\
\end{array}%
\right.
\end{equation*}
and $[\textbf{B}{(h,u,\varphi)}](0_1)=0_2$, where $I_h=\{ s\in S\mid h(s)=0_T\}$ is a closed ideal of $S$ and $0_1$ and $0_2$ are the zeros of the semigroups $B_{\lambda_1}^0(S)$ and
$B_{\lambda_2}^0(T)$, respectively. For each $(h,u,\varphi)\in\operatorname{\textbf{Mor}}(\mathscr{T\!PC\!B})$ with a trivial homomorphism $h$ we define a map $\textbf{B}{(h,u,\varphi)}\colon
\textbf{B}(S,{\lambda_1})=B_{\lambda_1}^0(S)\rightarrow \textbf{B}(T,{\lambda_2})=B_{\lambda_2}^0(T)$ as follows:
$[\textbf{B}{(h,u,\varphi)}](a)=0_2$ for all $a\in\textbf{B}(S,I_{\lambda_1})=B_{\lambda_1}^0(S)$. If $S$ is a trivial semigroup then we define $\textbf{B}(S,{\lambda_1})$ to be a trivial semigroup.

A functor $\textbf{F}$ from a category $\mathscr{C}$ into a
category $\mathscr{K}$ is called \emph{full} if for any
$a,b\in\operatorname{\textbf{Ob}}(\mathscr{C})$ and for any
$\mathscr{K}$-morphism $\alpha\colon \textbf{F}a\rightarrow
\textbf{F}b$ there exists a $\mathscr{C}$-morphism $\beta\colon
a\rightarrow b$ such that $\textbf{F}\beta=\alpha$, and
$\textbf{F}$ called \emph{representative} if for any
$a\in\operatorname{\textbf{Ob}}(\mathscr{K})$ there exists
$b\in\operatorname{\textbf{Ob}}(\mathscr{C})$ such that $a$ and
$\textbf{F}b$ are isomorphic \cite{Petrich1984}.

Theorem~4.1 \cite{GutikRepovs2010} and Theorem~\ref{theorem-3.2} imply

\begin{theorem}\label{theorem-3.3}
$\operatorname{\textbf{B}}$ is a full representative functor from $\mathscr{T\!PC\!B}$ into
$\mathscr{B}^*(\mathscr{T\!PC\!S})$.
\end{theorem}

\begin{remark}
We observe that the statements similar to Theorem~\ref{theorem-3.3} hold for the categories of semiregular countably pracompact (resp., Hausdorff countably compact,
Hausdorff sequentially compact, Hausdorff compact) semitopological monoids and nonempty sets and corresponding to them semiregular countably pracompact (resp., Hausdorff countably compact, Hausdorff sequentially compact, Hausdorff compact) topological Brandt $\lambda^0$-extensions of semiregular countably pracompact (resp., Hausdorff countably compact, Hausdorff sequentially compact, Hausdorff compact) semitopological  monoids in the class of semitopological semigroups. Moreover in the case of semiregular pseudocompact (resp., semiregular countably pracompact, Hausdorff countably compact, Hausdorff sequentially compact, Hausdorff compact) semitopological semilattices the functor $\operatorname{\textbf{B}}$ determines the equivalency of such categories and corresponding to them categories of semiregular pseudocompact (resp., semiregular countably pracompact, Hausdorff countably compact, Hausdorff sequentially compact, Hausdorff compact) topological Brandt $\lambda^0$-extensions of countably compact (resp., semiregular countably pracompact, Hausdorff countably compact, Hausdorff sequentially compact, Hausdorff compact) semitopological
semilattices with zero in the class of semitopological semigroups. The last assertion follows from Proposition~4.3~\cite{GutikRepovs2010}.
\end{remark}

Comfort and Ross \cite{ComfortRoss1966} proved that the Stone-\v{C}ech compactification of a pseudocompact topological group is a topological group. Therefore the functor of the Stone-\v{C}ech compactification \textbf{$\beta$} from the category of pseudocompact topological groups back into itself determines a
monad. Gutik and Repov\v{s} in \cite{GutikRepovs2007} proved the similar result for countably compact $0$-simple topological inverse semigroups: every countably compact $0$-simple topological inverse semigroup is topologically isomorphic to the topological Brandt $\lambda^0$-extension of a countably compact topological group with adjoined isolated zero for some finite non-zero cardinal $\lambda$ and the Stone-\v{C}ech compactification of a countably compact $0$-simple topological inverse semigroup is a compact $0$-simple topological inverse semigroup (see Theorems~2 and 3 in \cite{GutikRepovs2007}). Hence the functor of the Stone-\v{C}ech compactification
$\beta\colon \mathscr{B}^*(\mathscr{CC\!T\!G}) \rightarrow \mathscr{B}^*(\mathscr{CC\!T\!G})$ from the category of Hausdorff countably compact Brandt topological semigroups $\mathscr{B}^*(\mathscr{CC\!T\!G})$ into itself  determines a monad \cite{GutikRepovs2010}. Also, in \cite{GutikPavlykReiter2011} were obtained similar results to \cite{GutikRepovs2007} for pseudocompact completely $0$-simple topological inverse semigroups. There was proved that every countably compact completely $0$-simple topological inverse semigroup is topologically isomorphic to the topological Brandt $\lambda^0$-extension of a pseudocompact topological group with adjoined isolated zero for some finite non-zero cardinal $\lambda$ and the Stone-\v{C}ech compactification of a pseudocompact completely $0$-simple topological inverse semigroup is a compact $0$-simple topological inverse semigroup (see Theorems~1 and 2 in \cite{GutikPavlykReiter2011}).

By $\mathscr{B}^*(\mathscr{PC\!T\!G})$ we denote the category of completely $0$-simple Hausdorff pseudocompact topological inverse semigroup, i.e., the objects of $\mathscr{B}^*(\mathscr{PC\!T\!G})$ are Hausdorff pseudocompact Brandt topological inverse semigroups and morphisms in $\mathscr{B}^*(\mathscr{PC\!T\!G})$ are continuous homomorphisms between such topological semigroups. Then Theorems~2 from \cite{GutikPavlykReiter2011} implies the following:

\begin{corollary}\label{corollary-3.4}
The functor of the Stone-\v{C}ech compactification $\beta\colon
\mathscr{B}^*(\mathscr{PC\!T\!G}) \rightarrow
\mathscr{B}^*(\mathscr{PC\!T\!G})$ determines a monad.
\end{corollary}

%%%%%%%%%%%%%%%%%%%%%%%%%%%%%%%%%%%%%%%%%%%%%

\section*{Acknowledgements}

This research was carried out with the support of the Estonian Science Foundation and co-funded by Marie Curie Action, grant ERMOS36 
and the grant of National Academy of Sciences of Ukraine for young scientists, grant 0112U005004.

We acknowledge Alex Ravsky for his comments and suggestions.
%%%%%%%%%%%%%%%%%%%%%%%%%%%%%%%%%%%%%%%%%%%%%%%%%%%%%%%%%%%%


\begin{thebibliography}{22}


\bibitem{Arhangel'skij1984}
A.~V.~Arhangel'skij, \emph{Function spaces in the topology of
pointwise convergence,  and compact sets},  Uspekhi Mat. Nauk
\textbf{39}:5 (1984), 11--50 (in Russian); English version:  Russ.
Math. Surv. \textbf{39}:5 (1984), 9--56.

\bibitem{Arkhangelskii1992}
A.~V.~Arkhangel'skii, {\em Topological Function Spaces}, Kluwer
Publ., Dordrecht, 1992.

\bibitem{Clifford1942}
A.~H.~Clifford, \emph{Matrix representations of completely simple semigroups}, Amer. J. Math. \textbf{64} (1942), 327--342.

\bibitem{BucurDeleanu1968}
I.~Bucur and A.~Deleanu, \emph{Introduction to the Theory of Categories and Functors}, John Willey and Sons, Ltd., London, New York and Sidney, 1968.

\bibitem{CliffordPreston1961-1967}
A.~H.~Clifford and G.~B.~Preston, \emph{The Algebraic Theory of Semigroups}, Vol. I. Amer. Math. Soc. Surveys 7, 1961; Vol. II. Amer. Math. Soc. Surveys 7, 1967.

\bibitem{ComfortRoss1966}
W.~W.~Comfort and K.~A.~Ross, {\em Pseudocompactness and uniform continuity in topological groups}, Pacif. J. Math. {\bf 16} (1966), 483--496.

\bibitem{DeLeeuwGlicksberg1961}
K.~DeLeeuw, and I.~Glicksberg, \emph{Almost-periodic functions on semigroups}, Acta Math. {\bf 105} (1961), 99--140.


\bibitem{Engelking1989}
R.~Engelking, \emph{General Topology}, 2nd ed., Heldermann, Berlin,
1989.

\bibitem{Gutik1999}
O.~V.~Gutik, \emph{On Howie semigroup}, Mat. Metody Phis.-Mech. Polya. {\bf 42}:4 (1999), 127--132 (in Ukrainian).

\bibitem{GutikPavlyk2001}
O.~V.~Gutik and K.~P.~Pavlyk, \emph{$H$-closed topological semigroup and Brandt
$\lambda$-extensions}, Mat. Metody Phis.-Mech. Polya. {\bf 44}:3 (2001), 20--28 (in Ukrainian).

\bibitem{GutikPavlyk2005}
O.~V.~Gutik and K.~P.~Pavlyk, \emph{Topological semigroups of matrix units}, Algebra Discrete Math. no.~\textbf{3} (2005), 1--17.



\bibitem{GutikPavlyk2006}
O.~V.~Gutik and K.~P.~Pavlyk, \emph{On Brandt $\lambda^0$-extensions of semigroups with zero}, Mat. Metody Phis.-Mech. Polya. {\bf 49}:3 (2006), 26--40.

\bibitem{GutikPavlykReiter2009}
{O.~Gutik, K.~Pavlyk, and A.~Reiter,} {\em Topological semigroups of matrix units and countably compact Brandt $\lambda^0$-extensions}, Mat. Stud. \textbf{32}:2 (2009), 115--131.

\bibitem{GutikPavlykReiter2011}
O. V. Gutik, K. P. Pavlyk and A. R. Reiter, \emph{On topological Brandt semigroups}, Math. Methods and Phys.-Mech. Fields \textbf{54}:2 (2011), 7--16 (in Ukrainian); English Version in: J. Math. Sc. \textbf{184}:1 (2012), 1--11.

\bibitem{GutikRepovs2007}
O.~Gutik and D.~Repov\v{s}, \emph{On countably compact $0$-simple topological inverse semigroups}, Semigroup Forum \textbf{75}:2 (2007), 464--469.

\bibitem{GutikRepovs2010}
O.~Gutik and D.~Repov\v{s}, \emph{On Brandt $\lambda^0$-extensions of monoids with zero}, Semigroup Forum \textbf{80}:1 (2010), 8--32.

\bibitem{Howie1995} J. M. Howie,
\emph{Fundamentals of Semigroup Theory}, London Math. Monographs, New Ser. 12,  Clarendon Press, Oxford, 1995.

\bibitem{Munn1957}
W.~D.~Munn, \emph{Matrix representations of semigroups}, Proc. Cambridge Phil. Soc.
\textbf{53} (1957), 5--12.

\bibitem{Petrich1984}
M.~Petrich, \emph{Inverse Semigroups}, John Wiley $\&$ Sons, New York, 1984.

\bibitem{Reznichenko1994}
E.~A.~Reznichenko, \emph{Extension of functions defined on products of pseudocompact spaces  and continuity of the inverse in pseudocompact groups}, Topology Appl. \textbf{59}:3 (1994),  233–-244.

\bibitem{Ruppert1984}
W.~Ruppert, \emph{Compact Semitopological Semigroups: An Intrinsic Theory}, Lecture Notes in Mathematics, Vol.~1079, Springer, Berlin, 1984.



\end{thebibliography}
\end{document}